\documentclass[11pt,reqno]{amsart}

\pdfoutput=1
\usepackage[linkcolor=blue, citecolor=magenta, colorlinks]{hyperref}
\usepackage[protrusion=true,expansion=true]{microtype}
\usepackage{mathptmx}
\usepackage{amsthm}
\usepackage{amssymb}
\usepackage{verbatim}
\usepackage{stmaryrd}
\usepackage{amsmath,amsfonts}
\usepackage{mathrsfs}
\usepackage[all]{xy}
\usepackage{color}
\usepackage{fancyhdr}
\usepackage{array}
\usepackage[left=1.2in, right=1.2in, top=1in, bottom=1in]{geometry}
\usepackage{pb-diagram,pb-xy}
\usepackage{etoolbox}
\patchcmd{\section}{\scshape}{\bfseries}{}{}
\makeatletter
\renewcommand{\@secnumfont}{\bfseries}
\makeatother

\usepackage[backgroundcolor=yellow,linecolor=yellow,textsize=footnotesize]{todonotes} \makeatletter \providecommand \@dotsep{5} \def\listtodoname{List of Todos} \def\listoftodos{\@starttoc{tdo}\listtodoname} \makeatother 

\DeclareMathOperator{\Hom}{Hom}

\theoremstyle{plain}
\newtheorem{thmA}{Theorem}

\newtheorem{theorem}{Theorem}[section]  
\newtheorem{lemma}[theorem]{Lemma}

\newtheorem{proposition}[theorem]{Proposition}

\theoremstyle{definition}
\newtheorem{definition}[theorem]{Definition}
\newtheorem{example}[theorem]{Example}

\newtheorem{remark}[theorem]{Remark}
\newtheorem{construction}[theorem]{Construction}

\newcommand{\sets}{\mathbf{Sets}}
\newcommand{\hyperrings}{\mathbf{Hyperrings}}
\newcommand{\hyperfields}{\mathbf{Hyperfields}}
\newcommand{\parthyperrings}{\mathbf{PartHyperrings}}
\newcommand{\parthyperfields}{\mathbf{PartHyperfields}}
\newcommand{\ddhyperfields}{\mathbf{Hyperfields}_{dd}}
\newcommand{\fuzzyweak}{\mathbf{FuzzyRings}_{wk}}
\newcommand{\fuzzyweakast}{\mathbf{FuzzyRings}_{wk}^{\ast}}
\newcommand{\fuzzystrong}{\mathbf{FuzzyRings}_{str}}
\newcommand{\partialdemifields}{\mathbf{PartialDemifields}}
\newcommand{\functor}{\mathcal{F}}
\newcommand{\invfunctor}{\mathcal{G}}
\newcommand{\reducedfunctor}{\overline{\mathcal{F}}}
\newcommand{\demifunctorA}{\mathcal{F}_1}
\newcommand{\demifunctorB}{\mathcal{F}_2}

\begin{document}

\title{On the relation between hyperrings and fuzzy rings}
\author{Jeffrey Giansiracusa}
\address{Department of Mathematics, Swansea University, Singleton Park, Swansea SA2 8PP, UK}
\email{j.h.giansiracusa@swansea.ac.uk}

\author{Jaiung Jun}
\address{Department of Mathematical Sciences, Binghamton University, Binghamton, NY 13902, USA}
\email{jjun@math.binghamton.edu}

\author{Oliver Lorscheid}
\address{Instituto Nacional De Matem\'{a}tica Pura A Aplicada, Estrada Dona Castorina 110, Rio De Janeiro, Brazil}
\email{oliver@impa.br
}



\date{\today}

\begin{abstract}
  We construct a full embedding of the category of hyperfields into Dress's category of fuzzy rings
  and explicitly characterize the essential image --- it fails to be essentially surjective in a
  very minor way.  This embedding provides an identification of Baker's theory of matroids over
  hyperfields with Dress's theory of matroids over fuzzy rings (provided one restricts to those
  fuzzy rings in the essential image).  The embedding functor extends from hyperfields to
  hyperrings, and we study this extension in detail.  We also analyze the relation between
  hyperfields and Baker's partial demifields.
\end{abstract}

\maketitle

\section{Introduction}

The important and pervasive combinatorial notion of matroids has spawned a number of variants over
the years.  In \cite{dress1986duality} and \cite{dress1991grassmann}, Dress and Wenzel developed a
unified framework for these variants by introducing a generalization of rings called \emph{fuzzy
  rings} and defining matroids with coefficients in a fuzzy ring.  Various flavors of matroids,
including ordinary matroids, oriented matroids, and the valuated matroids introduced in
\cite{dress1992valuated}, correspond to different choices of the coefficient fuzzy ring. 

Roughly speaking, a fuzzy ring is a set $S$ with single-valued unital addition and multiplication
operations that satisfy a list of conditions analogous to those of a ring, such as distributivity,
but only up to a tolerance prescribed by a distinguished ideal-like subset $S_0$.  Beyond the work
of Dress and Wenzel, fuzzy rings have not yet received significant attention in the literature.  A
somewhat different generalization of rings, known as \emph{hyperrings}, has been around for many
decades and has been studied very broadly in the literature.  Roughly, a hyperring is a set $R$ with
a single-valued multiplication operation $\times_R$ and a multi-valued addition operation $+_R$
satisfying a list of conditions that are analogous to the defining conditions of a ring.  A
\emph{hyperfield} is a hyperring in which the multiplicative monoid of nonzero elements is a group.

In the recent elegant paper \cite{baker2016matroids}, Baker defined and studied matroids with
coefficients in a hyperfield.  As with fuzzy rings, many common flavors of matroids correspond to
appropriate choices of coefficients.  However, beyond this, Baker shows that hyperfields provide a
compelling setting for matroid theory --- one to which duality theory and many of the most common
cryptomorphic presentations of matroids all extend (including the circuit and dual pair axioms,
which are absent in the work of Dress and Wenzel), and one which is built on arguably simpler and
more familiar algebraic structures.

In view of these two approaches toward a unified theory of matroids, a natural question to ask is
how hyperrings and fuzzy rings are related, and how their corresponding theories of matroids are
related.  These are precisely the questions we set out to address in this paper.  The full story is
somewhat messy, and one of the purposes of this paper is to map the terrain.  However, at the center
of things there is a clean statement: \textbf{there is a fully faithful functor from hyperfields to
  fuzzy rings, and it fails to be essentially surjective in a very mild way that we explicitly
  describe.  Moreover, this embedding induces an identification between matroids over a hyperfield
  and matroids over the corresponding fuzzy ring.}

The ideas behind this result follow from three observations: (1) the multi-valued addition operation
of a hyperring can be viewed as a single-valued operation taking values in the powerset of the
hyperring, but the cost of this is that the resulting object is a fuzzy ring rather than a ring, (2)
one of Dress's notions of morphisms for fuzzy rings is quite flexible and provides a category in which
distinct fuzzy rings can nevertheless be isomorphic, and (3) the definition of matroids over fuzzy
rings sees only the multiplicative units.

\subsection{Statement of results in detail}

We will consider the category $\hyperrings$ of hyperrings and the subcategory $\hyperfields$
consisting of hyperfields.  On the other side, fuzzy rings come with two distinct notions of
morphisms, originally called \emph{morphisms} and \emph{homomorphisms} by Dress, but we prefer to
call them \emph{weak morphisms} and \emph{strong morphisms} as they have less potential for
confusion, and we will denote the corresponding categories of fuzzy rings by $\fuzzyweak$ and
$\fuzzystrong$.  
Since strong morphisms of fuzzy rings restrict to weak morphisms, we obtain a functor $\fuzzystrong \to \fuzzyweak$, which turns out to be neither full nor faithful. All of these definitions will be reviewed in section \ref{sec:review}.

\subsubsection{Main results}

\begin{thmA}\label{thm:main-result}
  There is fully faithful functor
  \[
  \functor:\hyperfields \hookrightarrow \fuzzyweak,
\]
that commutes with the respective forgetful functors to abelian groups given by sending a hyperfield
and a fuzzy ring to its multiplicative group of units, i.e., there are natural isomorphisms
$F^\times\cong\functor(F)^\times$.  Moreover, the essential image consists of those fuzzy rings
$(K,+,\times,\epsilon, K_0)$ such that for each pair of units $a,b\in K^\times$, there exists $c\in
K^\times \cup \{0\}$ such that $a+b+c \in K_0$.
\end{thmA}

In \cite{dress1991grassmann}, Dress and Wenzel showed that the original definition of matroids with
coefficients in a fuzzy ring is cryptomorphic to a fuzzy ring version of the Grassmann-Pl\"ucker
axiom system.  Sending a hyperfield $F$ to the set of rank $r$ matroids on a finite set $E$ with
coefficients in $F$ (in the sense of \cite[\S 3]{baker2016matroids}), and using the pushforward
construction on morphisms, defines a functor
\[
\mathscr{M}_{E,r}^{\textit{hyp}}: \hyperfields \to \sets.
\]
Similarly, sending a fuzzy ring $K$ to the set of rank $r$ matroids on $E$ with coefficients in $K$
defines a functor
\[
\mathscr{M}_{E,r}^{\textit{fuzzy}}: \fuzzyweak \to \sets,
\]
(\textit{a priori} $\mathscr{M}_{E,r}^{\textit{fuzzy}}$ is defined on $\fuzzystrong$, but from the perspective of
Grassmann-Pl\"ucker functions it is immediate that it factors through $\fuzzystrong \to \fuzzyweak$).

\begin{thmA}\label{thm:matroids-result}
The diagram
\[
\begin{diagram}
\node{\hyperfields} \arrow{se,b}{\mathscr{M}_{E,r}^{\textit{hyp}}} \arrow[2]{e,t}{\functor} \node[2]{\fuzzyweak} \arrow{sw,b}{\mathscr{M}_{E,r}^{\textit{fuzzy}}} \\
\node[2]{\sets}
\end{diagram}
\]
commutes up to a natural isomorphism implemented by applying the identification
$F^\times \cong \functor(F)^\times$ to the components of the Grassmann-Pl\"ucker function.  In
particular, there is a bijection between matroids over a hyperfield $F$ and matroids over the
corresponding fuzzy ring $\functor(F)$.
\end{thmA}

\subsubsection{From hyperrings to fuzzy rings}
The functor giving the the embedding of Theorem \ref{thm:main-result} can be extended from
hyperfields to hyperrings, and it can be refined to take values in strong morphisms rather than just
weak morphisms. I.e., we actually have a functor (which we denote by the same symbol)
\[
\functor: \hyperrings \to \fuzzystrong
\]
that generalizes a construction of Dress that embeds the category of (ordinary) rings into fuzzy
rings. 

Let $(R,+,\times)$ be a hyperring.  The fuzzy ring $\functor(R)$ has the nonempty subsets of $R$ as its
elements, multiplication operation given by
\[
A \times_{\functor(R)} B = \{a \times b \mid a\in A \text{ and } b\in B\},
\]
for subsets $A,B\subset R$, and addition operation given by $A +_{\functor(R)} B  = \bigcup_{a\in A,
  b\in B} a + b$.   Note that sending $x \in R$ to the singleton $\{x\}\in \functor(R)$ gives an injective
multiplicative map $R\to \functor(R)$ that turns out to be an isomorphism on units. 

Going beyond Theorem \ref{thm:main-result}, we study the properties of the refined/extended version
of $\functor$.

\begin{thmA} 
  The functor $\functor: \hyperrings \to \fuzzystrong$ has the following properties.
\begin{enumerate}
\item It is faithful but neither full nor essentially surjective, and the same is true for its
  restriction to the full subcategory $\hyperfields$.
\item The composition 
\[\hyperrings \stackrel{\functor}{\to} \fuzzystrong \to \fuzzyweak
\]
 is faithful but not full. 
\end{enumerate}
\end{thmA}

\subsubsection{Doubly-distributive hyperfields and valuative hyperfields}

A hyperfield is said to be \emph{doubly distributive} if it satisfies the condition
$(a+b)(c+d) = ac + ad + bc + bd$.  On the subcategory $\ddhyperfields \subset \hyperfields$ of
doubly-distributive hyperfields and strict morphisms (meaning the containment
$f(a+b) \subset f(a)+f(b)$ is an equality) there is a reduced version,
\[
\reducedfunctor: \ddhyperfields \to \fuzzystrong,
\]
of the functor $\functor$ studied above. Given a doubly distributive hyperfield $H$, the fuzzy ring
$\reducedfunctor(H)$ is always a sub-fuzzy ring of $\functor(H)$ (meaning that there is a strong
morphism that is set-theoretically injective), and the inclusion restricts to an isomorphism on
units, so $H^\times \cong \functor(H)^\times \cong \reducedfunctor(H)^\times$.

\begin{thmA}
  The functor $\reducedfunctor$ is faithful but neither full nor essentially surjective, as is its
  composition with the functor $\fuzzystrong \to \fuzzyweak$.  If $H$ is a doubly-distributive
  hyperfield then, as in Theorem \ref{thm:matroids-result} above, under the identification
  $H^\times \cong\reducedfunctor(H)^\times$ applied to the components of a Grassmann-Pl\"ucker
  function, there is a bijection between matroids with coefficients in $H$ and matroids with
  coefficients in the fuzzy ring $\reducedfunctor(H)$.
\end{thmA}

Given a totally ordered abelian group $\Gamma$, Dress and Wenzel \cite[p.~237]{dress1992valuated}
construct a fuzzy ring $K_\Gamma$ such that matroids with coefficients in this fuzzy ring are
precisely $\Gamma$-valuated matroids.  On the other hand, there is a hyperfield $H_\Gamma$ whose
underlying set is just $\Gamma\cup \{0\}$ (see e.g., \cite[Example 2.9]{baker2016matroids}), and
this hyperfields turns out to be doubly-distributive. In \S 5, we turn to the relation between the
fuzzy ring $K_\Gamma$ and the hyperfield $H_\Gamma$.

\begin{thmA}
There is a canonical isomorphism $\reducedfunctor(H_\Gamma) \cong K_\Gamma$.
\end{thmA}

\subsubsection{Partial demifields and doubly-distributive hyperfields}

A \emph{partial demifield} is an additional structure that can be placed on a hyperfield.  This
structure was introduced by Baker in \cite[Definition 4.1]{baker2016matroids} in order to explain
why matroids over some hyperfields admit a description in terms of vector axioms, but others do not
seem to.  It can also be seen as a generalization and refinement of the notion of \emph{patial
  fields} from \cite{pendavingh2010lifts}.  Not all hyperfields admit a partial demifield structure,
and when such a structure exists it need not be unique.

Let $\partialdemifields$ be the category of partial demifields.

\begin{thmA}\label{thm:dd-factorization}
The functor $\reducedfunctor$ admits a factorization $\reducedfunctor = \demifunctorB \circ
\demifunctorA$ where \[\demifunctorA: \ddhyperfields \to \partialdemifields\]
 and
  \[
\demifunctorB: Im(\demifunctorA) \to \fuzzystrong.
\]
The functor $\demifunctorA$ is neither full nor essentially surjective. 
\end{thmA}

In particular, this says that doubly-distributive hyperfields admit canonical (though not
necessarily unique) partial demifield structures in a functorial way. Note also that since
$\reducedfunctor$ is faithful, each of $\demifunctorA$ and $\demifunctorB$ is also faithful.


\subsection*{Acknowledgments} We thank Matt Baker writing the paper that motivated this work and for
providing encouragement and useful comments on a draft.

\section{Review of hyperrings and fuzzy rings}\label{sec:review}

\subsection{A brief history}

The study of multi-valued binary operations can be traced back to Marty \cite{marty1935role}, where
the notion of hypergroups first appeared.  In \cite{krasner1956approximation} Krasner introduced
hyperrings in order to study a problem involving approximating a complete valued field of positive
characteristic by a sequence of such fields in characteristic zero.  For some time interest in
hyperrings was mainly concentrated in certain applied areas (see, e.g.,
\cite{corsini2003applications} and \cite{Dav2}) and hyperrings received relatively little attention
in the pure mathematics community. However, in the last decade that has changed.  For example, Viro
argues in \cite{viro}, \cite{viro2} that hyperrings are a natural structure deserving of study and in
particular they should play a central role in tropical geometry, Gladki and Marshall relate
hyperring theory to abstract real spectra and quadratic forms (\cite{mars1}, \cite{gladki2016witt}),
Connes and Consani \cite{con3} use hyperrings in their work on the Ad\`{e}le class space
$\mathbb{A}_K/K^\times$ over a global field $K$. The second author develops in \cite{jun2015algebraic} a scheme theory based on hyperrings.
Baker \cite{baker2016matroids} has shown that
hyperfields provide an elegant tool for unifying several different flavors of matroid theory.

Fuzzy rings are a different way of relaxing the axioms of rings.  They were introduced by Dress
in \cite{dress1986duality} to provide a unified framework for various classes of matroids, and they
have been developed further in subsequent papers by Dress and Wenzel.  The concept appears to have
not yet been taken up more broadly in the community.

\subsection{Hyperrings}

In this section we briefly review the definition of a hyperring.  For more details, we refer the
reader to \cite{jun2015algebraic} or \cite{baker2016matroids}.

Given a set $A$, we write $\mathcal{P}^*(A)$ for the set of all nonempty subsets of $A$.  A binary
\emph{hyperoperation} on $A$ is a function $+: A \times A \longrightarrow \mathcal{P}^*(A)$.  Note
that a binary operation can be considered as a hyperoperation valued in singletons.

\begin{construction}\label{con:op-ext}
Any binary hyperoperation $+$ on $A$ determines an associated binary operation on the powerset
$\mathcal{P}^*(A)$ by the formula
\begin{equation}
  X+Y=\bigcup_{x\in X, y\in Y} (x+y)
\end{equation} 
for $X,Y\subset A$.  
\end{construction}

Throughout this paper we shall make extensive use of the above construction.  It is
used in formulating what it means for a hyperoperation to be associative, as in the definition
below, and it forms the basis for the definition of the functors $\functor$ and $\reducedfunctor$
that are the central objects of study in this paper.

\begin{definition}\label{canonicalhypergroup}
  A \textbf{canonical hypergroup} is a set $A$ together with a hyperoperation $+$ (sometimes called a
  hyperaddition) that satisfies the following conditions:
  \begin{enumerate}
  \item (Commutativity) $a+b=b+a$ for all $a,b\in A$,
  \item (Identity) There is a (necessarily unique) element $0$ in $A$ such that $a+0=\{a\}$ for all
    $a \in A$.
  \item (Inverses) For each $a \in A$, there exists a unique element $-a \in A$ such that
    $0 \in a+ (-a)$.
  \item (Associativity) $(a+b)+c = a+(b+c)$ for all $a,b,c\in A$.
  \item (Reversibility) $a \in b+c$ if and only if $c \in a+(-b)$, for all $a,b,c\in A$.
  \end{enumerate}
\end{definition}

Observe that if a hyperoperation on $A$ is either commutative or associative then the associated
binary operation on $\mathcal{P}(A)^*$ is as well.

\begin{definition}
  A \textbf{hyperring} is a set $R$ with a multiplication operation $\times$ and an addition hyperoperation
  $+$ such that $(R,+)$ is a canonical hypergroup, $(R,\times)$ is a unital commutative monoid, and
  $+$ and $\times$ satisfy the conditions:
\begin{enumerate}
\item $a\times (b+c) = (a\times b) + (a \times c)$  for all $a,b,c \in R$, and
\item $a \times 0 = 0$ for all $a \in R$.
\end{enumerate}
When $(R \smallsetminus \{0\},\times)$ is a group the hyperring $(R,+,\times)$ is said to be a
\textbf{hyperfield}.
\end{definition}

As is usual practice, when there is no risk of ambiguity, we will often abuse notation and refer to
hyperrings by their underlying sets, writing $R$ in place of $(R,+,\times)$.  We will also commonly
write $ab$ for the product $a\times b$.

\begin{definition}
  Let $R$ and $S$ be hyperrings.  A \textbf{homomorphism} from $R$ to $S$ is a function
  $f: R\to S$ such that $f(0_R)=0_S$, $f(1_R)=1_S$, and for all $a,b\in R$ 
\begin{equation}\label{mor}
  f(a+_Rb) \subseteq f(a)+_S f(b), \quad \text{and} \quad f(a\times_R b)=f(a)\times_S f(b). 
\end{equation}
If the containment of \eqref{mor} is an equality for all $a,b \in R$ then $f$ is said to be a
\textbf{strict homomorphism.}
\end{definition}

\begin{definition}
  Let $F$ be a hyperfield. We say that $F$ is \textbf{doubly-distributive} if, for any
  $a,b,c,d \in F$, one has
\begin{equation}\label{doubly}
(a+b)(c+d)=ac+ad+bc+bd. 
\end{equation}
\end{definition}

\begin{remark}
Not all hyperfields are doubly-distributive.  See for example \cite[Theorem 5.B.]{viro} or Example \ref{dddd}
\end{remark}

We list four examples of hyperfields which are used by Baker in \cite{baker2016matroids} to unify
various flavors of matroids.

\begin{example}[The Krasner hyperfield] 
  Let $\mathbb{K}$ be the set $\{0,1\}$ equipped with the same multiplication as the field
  $\mathbb{F}_2$ and with the commutative addition hyperoperation given by $1+0=1$, $0+0=0$, and
  $1+1=\{0,1\}$.  The resulting structure is a hyperfield called the \textbf{Krasner hyperfield.}
\end{example}

\begin{example}[The hyperfield of signs]
  Let $\mathbb{S}:=\{-1,0,1\}$.  The multiplication on $\mathbb{S}$ is the restriction of
  multiplication on the integers.  The hyperaddition is given by
\begin{align*}
 1+1=1+0=\{1\},\\
  (-1)+(-1)=(-1)+0=\{-1\},\\
 1+(-1)=\{-1,0,1\}.
\end{align*}
The resulting hyperfield structure on $\mathbb{S}$ is called the \textbf{hyperfield of signs}.
\end{example}

\begin{example}[The tropical hyperfield]
  Let $\mathbb{T}:=\mathbb{R}\cup\{-\infty\}$ equipped with the ordering where $\mathbb{R}$ has the
  usual order and $-\infty$ is minimal.  Multiplication is given by the usual addition of real
  numbers together, extended so that $-\infty$ is absorbing. The hyperaddition is given by
  $a + b = \max(a,b)$ if $a \neq b$, and $a + a =[-\infty,a]$.  With this structure $\mathbb{T}$
  becomes a hyperfield called the \textbf{tropical hyperfield}.
\end{example}

\begin{example}[The phase hyperfield]
  Let $\mathbb{P}:=\mathbb{S}^1 \cup\{0\} \subset \mathbb{C}$.  Multiplication is the restriction of
  the usual multiplication of complex numbers, and we define a hyperaddition by
  $x + (-x):=\{-x,0,x\}$ and
  $x +  y:= \{\frac{ax+by}{||ax+by||}\mid a,b \in \mathbb{R}_{>0}\}$ if $x+y \neq 0$. $\mathbb{P}$.
  The resulting structure is a hyperfield called the \textbf{phase hyperfield}.
\end{example}

A useful source of examples of hyperrings is the following quotient construction,  which is proved
to yield a hyperring structure in \cite[Proposition 2.6.]{con3}.

\begin{construction}\label{quotientconstruction}
Let $R$ be a commutative ring, let $R^\times$ be the multiplicative group of units in $R$, and  let
$U$ be a subgroup of $R^\times$.  We consider $U$ acting on $R$ by multiplication and construct a
hyperring structure on the set of cosets $R/U$.  The hyperaddition $\oplus$ is defined by
\[
[a]\oplus [b] := \{ [c] \in R/U \mid c \in aU + bU\}. 
\]
and the multiplication is defined by
\[
[a]\odot [b]:=[ab]. 
\]
\end{construction}

\begin{example}
\begin{enumerate}
\item Let $k$ be any field such that $|k| \geq 3$ and $U=k^\times$. Then the Krasner hyperfield
  $\mathbb{K}$ is isomorphic to the hyperfield $k/U$.
\item Let $\mathbb{Q}$ be the field of rational numbers and $U=\mathbb{Q}_{>0}$ be the subgroup of
  the multiplicative group $\mathbb{Q}^\times$ which consists of positive rational numbers. Then one
  can easily see that $\mathbb{Q}/U$ is isomorphic to the hyperfield of signs $\mathbb{S}$.
\end{enumerate}
\end{example}

\begin{remark}
  There is a beautiful one-to-one correspondence between hyperfield extensions of $\mathbb{K}$ and
  projective geometries (with the condition that each line contains at least four points) together
  with a group of collineations. This correspondence links the complete classification problem of
  hyperfield extensions of $\mathbb{K}$ to (the abelian case of) a long-lasting conjecture on the
  existence of finite Non-Desarguesian projective plane with a simply transitive automorphisms. We
  also note that there exists a similar result relating hyperfield extensions of $\mathbb{S}$ and
  spherical geometries. For more details, see \cite[\S 3]{con3}.
\end{remark}

\subsection{Fuzzy rings and their morphisms}
We review the definition of fuzzy rings first introduced by Dress in \cite{dress1986duality}. 

\begin{definition}$($\cite[\S1]{dress1986duality} $)$ A \textbf{fuzzy ring} is a tuple
  $(K;+,\times;\varepsilon,K_0)$ where $K$ is a set equipped with two binary operations $+$ and
  $\times$, $\varepsilon\in K$ is a distinguished element, and $K_0 \subset K$ is a specified
  subset, subject to the following conditions:
  \begin{itemize}
  \item[(FR0)] $(K,+)$ and $(K,\times)$ are commutative monoids with neutral elements $0$ and $1$.
  \item[(FR1)] ($0$ is an absorbing element) $0 \times a=0$ for all $ a \in K$.
  \item[(FR2)] (Units distribute over addition) $a(b+c)=ab +ac$ for all $b,c \in K$ and $a \in K^\times$ (the
    group of multiplicative units of $K$).
  \item[(FR3)] $\varepsilon^2=1$.
  \item[(FR4)] $K_0+K_0 \subseteq K_0$, $K\times K_0 \subseteq K_0$, $0 \in K_0$ and $1 \notin K_0$.
  \item[(FR5)] For $a \in K^\times$, $(1+a) \in K_0$ if and only if $a=\varepsilon$.
  \item[(FR6)] For $a,b,c,d \in K$, if $(a+b)$ and $(c+d)$ are both in  $K_0$ then
    $ac + \varepsilon b d \in K_0$.
  \item[(FR7)] For $a,b,c,d \in K$, if $a+b(c+d) \in K_0$ then $a + bc + bd \in K_0$.
  \end{itemize}
\end{definition}
We will call an element \textbf{null} if it lies in $K_0$.

\begin{remark}
 These axioms can be reformulated in a slightly more intuitive manner as follows. The set $K$ together with the addition $+$ and the multiplication $\times$ forms a \emph{(non distributive) commutative semiring with $0$ and $1$}, by which we mean the validity of axioms (FR0) and (FR1). The subset $K_0\subset K$ is a \emph{proper semiring ideal}, by which we mean the validity of (FR4).
 
 The element $\epsilon$ is determined uniquely by property (FR5), which allows us to eliminate this constant from the language. Thus (FR3) and (FR5) are equivalent to saying that there exists a unique \emph{fuzzy inverse of $1$ with (respect to $K_0$)}, which is an element $\epsilon\in K^\times$ such that $1+\epsilon\in K_0$; this fuzzy inverse of $1$ satisfies (FR3) and (FR6).
 
 The semiring $K$ satisfies a \emph{fuzzy distributivity (with respect to $K_0$)} in the sense of (FR2) and (FR7).
\end{remark}

\begin{definition}[ {\cite[\S 1]{dress1986duality}}] Let $(K;+,\times, \varepsilon_K, K_0)$ and
  $(L; +, \times, \varepsilon_L, L_0)$ be fuzzy rings. A \textbf{weak morphism} $K \to L$ is a group
  homomorphism $f: (K^\times,\times) \to (L^\times,\times)$ which satisfies the following condition: 
  \begin{equation}\label{fuzzymor}
    \text{For any }  a_1,\ldots,a_n \in K^\times, \text{ if } \sum_{i=1}^n a_i \in K_0 \text{ then } \sum_{i=1}^n f(a_i) \in L_0. 
  \end{equation}
\end{definition}

\begin{definition}[ {\cite[\S 1]{dress1986duality}}] \label{def:strong-morphism}Let $(K;+,\times, \varepsilon_K, K_0)$ and
  $(L; +, \times, \varepsilon_L, L_0)$ be fuzzy rings. A \textbf{strong morphism} $K \to L$ is a
  function $f: K \longrightarrow L$ with $f(1)=1$, $f(0)=0$, and such that:
\begin{enumerate}
\item \label{modulecondition}
  $f(a\times b)=f(a)\times f(b)$ for all  $a \in K^\times$ and  $b \in K$.
\item \label{concon}  Given $a_i,b_i \in K$, if $\sum_{i=1}^n a_i \times b_i \in K_0$ then $\sum_{i=1}^n
  f(a_i)\times f(b_i) \in L_0$.
\end{enumerate}
\end{definition}

\begin{remark}
  These two notions of morphisms were introduced in \cite{dress1986duality}, where they were called
  \emph{morphisms} and \emph{homomorphisms}.  We feel this terminology is potentially confusing, and
  so we prefer to call them weak and strong, as above.
\end{remark}

Note that if $f: (K;+,\times, \varepsilon_K, K_0) \to (L; +, \times, \varepsilon_L, L_0)$ is either
a strong or a weak morphism then it follows from the definitions that
$f(\varepsilon_K)=\varepsilon_L$. Indeed, suppose that $f$ is either a weak or a strong morphism.
Since $1+\varepsilon_K \in K_0$, we should have $f(1)+f(\varepsilon_K) \in L_0$. However, since
$f(1)=1$, this is equivalent $1+f(\varepsilon_K) \in L_0$ and hence $f(\varepsilon_K)=\varepsilon_L$
from $(\mathrm{FR5})$.

It is straightforward to check that the classes of strong and weak morphism are both closed under
composition.  Moreover, the identity on $K^\times$ is a weak morphism that is the identity with respect
to composition, and the identity on $K$ is a strong morphism that is the identity with respect to
composition of strong morphisms. Hence fuzzy rings with either strong or weak morphisms form
categories; we write $\fuzzystrong$ for the category of fuzzy rings with strong morphisms, and
$\fuzzyweak$ for the category of fuzzy rings with weak morphisms.

\begin{example}\label{krasnerfuzzy}
Let $K:=\{0,1,k_0\}$, $0=\{0\}$, $\varepsilon=1$, and $K_0=\{0,k_0\}$. Consider the following addition and multiplication: 
\[
\begin{tabular}{ | c | c | c |c| }
    \hline
    $+$ &$0$ & $1$ & $k_0$  \\ \hline
    $0$ & $0$ & $1$ & $k_0$ \\  \hline
    $1$ &$1$ & $k_0$ & $k_0$  \\ \hline
    $k_0$ & $k_0$ & $k_0$ & $k_0$ \\  
   
    \hline
    \end{tabular} \qquad 
    \begin{tabular}{ | c | c | c |c| }
        \hline
        $\times$ &$0$ & $1$ & $k_0$  \\ \hline
        $0$ & $0$ & $0$ & $0$ \\  \hline
        $1$ &$0$ & $1$ & $k_0$  \\ \hline
        $k_0$ & $0$ & $k_0$ & $k_0$ \\ 
       
        \hline
        \end{tabular}
\]

Then $K$ becomes a fuzzy ring. Moreover, $K$ is in fact a final object in both $\fuzzystrong$ and
$\fuzzyweak$. For more details, see \cite[\S 1.5]{dress1986duality}.
\end{example}

\begin{example} \label{signfuzzy}
Let $L:=\{0,1,-1, k_0\}$, $0=\{0\}$, $\varepsilon=-1$, and $L_0=\{0,k_0\}$. Consider the following addition and multiplication: 
\[
\begin{tabular}{ | c | c | c |c|c| }
    \hline
    $+$ &$0$ & $1$ & $-1$ &$k_0$  \\ \hline
    $0$ & $0$ & $1$ & $-1$ &$k_0$ \\  \hline
    $1$ & $1$ & $1$ & $k_0$ &$k_0$ \\  \hline
    $-1$ &$-1$ & $k_0$ &$-1$ & $k_0$  \\ \hline
    $k_0$ & $k_0$ & $k_0$ &$k_0$ & $k_0$ \\  
   
    \hline
    \end{tabular} \qquad 
    \begin{tabular}{ | c | c | c |c|c| }
        \hline
        $\times$ &$0$ & $1$ & $-1$ &$k_0$  \\ \hline
        $0$ & $0$ & $0$ & $0$ &$0$ \\  \hline
        $1$ & $0$ & $1$ & $-1$ &$k_0$ \\  \hline
        $-1$ &$0$ & $-1$ &$1$ & $k_0$  \\ \hline
        $k_0$ & $0$ & $k_0$ &$k_0$ & $k_0$ \\  
       
        \hline
        \end{tabular}
\]
Then $L$ is a fuzzy ring with the given addition and multiplication. 
\end{example}

\begin{remark}
  Examples \ref{krasnerfuzzy} and \ref{signfuzzy} can be obtained by using the quotient construction
  of fuzzy rings (see, \cite[\S 1.4]{dress1986duality}) and the quotient construction of fuzzy rings
  is similar to the construction of quotient hyperrings (see, \cite[\S 2]{jun2015algebraic}).
\end{remark}

\begin{proposition}\label{restriction}
  Restricting each strong morphism $f:K \to L$ to the group of units of $K$ determines a functor
  $\fuzzystrong \to \fuzzyweak$.
\end{proposition}
\begin{proof}
This follows immediately from the definitions by taking the $b_i$ to all be 1 in condition (2) of
the definition of strong morphism.
\end{proof}

The above functor is neither full nor faithful. The following example illustrates the lack of
faithfulness, and Example \ref{notfull} illustrates the failure to be full.

\begin{example}
  Any commutative ring $(R,+,\cdot)$ can be considered as a fuzzy ring by letting $K=R$,
  $K_0=\{0\}$, and $\varepsilon=-1$. This gives a faithful embedding of rings into fuzzy rings with
  strong morphisms.  In particular, consider the fuzzy ring $K$ associated with the univariate
  polynomial ring $\mathbb{Z}[x]$.  Sending $x$ to any polynomial $p(x)\in \mathbb{Z}[x]$ defines a ring
  endomorphism and hence a strong fuzzy ring endomorphism.  However, since the group $K^\times$ of units
  in $K$ is $\{1,-1\}$, all strong endomorphisms necessarily induce the same weak endomorphism,
  namely the identity.  
\end{example}

\section{The functor \texorpdfstring{$\functor: \hyperrings \to \fuzzystrong$}{from hyperrings to fuzzy rings}} \label{functor}

In \cite[Example 1.3]{dress1986duality}, Dress constructed a functor from rings to fuzzy rings.
We will first show that Dress's functor easily extends to hyperrings. 

Given a hyperring $R = (R,+, \times)$, consider the
following data:
\[
K = \mathcal{P}^*(R), \quad K_0 =\{T \in K \mid 0_R \in T\}, \quad \varepsilon = \{-1_R\} \in K,
\quad 0=\{0_R\} \in K_0, \quad 1=\{1_R\}\in K.
\]
The hyperaddition and multiplication on $R$ induce binary operations $+$ and $\times$ on $K$ via
Construction \ref{con:op-ext}.

\begin{definition}
  Let $\functor(R)$ denote the tuple $(K, +, \times, \epsilon, K_0)$ constructed from the hyperring
  $R$ as above.
\end{definition}

 In order to show that this yields a fuzzy ring, we first need the following lemma.

\begin{lemma}\label{units-are-singletons}
  With $K$ as above, the group of multiplicative units in $K$ consists of the singletons $\{x\}$
  such that $x$ is a multiplicative unit in $R$.
\end{lemma}
\begin{proof}
Suppose that $T_1,T_2 \in K^\times$ and $T_1\times T_2 =\{1_R\}$. This immediately implies that $T_i
\subseteq R^\times$ for $i=1,2$. If $x,y \in T_1$, then $x^{-1}, y^{-1} \in T_2$ since
$T_1\times_KT_2=\{1_R\}$. But this implies that $y\times x^{-1} =1_R$, so $x=y$.
\end{proof}




\begin{theorem}\label{object}
  Given a hyperring $R$, the tuple $\functor(R)$ constructed above is a fuzzy ring.
\end{theorem}
\begin{proof}
We will check the fuzzy ring axioms one by one.
\begin{enumerate}
\item[(FR0)] It follows directly from the definition of hyperrings that $(K,+)$ and $(K,\times)$ are commutative monoids.
\item[(FR1)]  One has $0\times a = 0$ for all $a \in K$ since $0=\{0_R\}$ and $0_R$ is an absorbing
  element in $R$.
\item[(FR2)] This condition follows from the distributive property of hyperrings and Lemma
  \ref{units-are-singletons}.  Suppose $T_1 \in K^\times$ and $T_2,T_3 \in K$.  Then $T_1=\{z\}$ for
  some $z \in R^\times$.  Then
\[
T_1 \times (T_2 +T_3)=\{z\} \times (T_2+T_3) = (\{z\}\times T_2)+(\{z\}\times T_3)=T_1\times T_2 +T_1 \times T_3. 
\]
\item[(FR3)] Since $\varepsilon= \{-1_R\}$, we have $\varepsilon^2=\{-1_R\}\times \{-1_R\}=\{1_R\}=1$.
\item[(FR4)] If $A,B \in K_0$ then $0_R \in A$ and $0_R \in B$, and hence
  $0_R \in A+B$, in particular, $K_0 + K_0 \subseteq K_0$. Also, clearly we have
  $K\times K_0 \subseteq K_0$, $0 \in K_0$ and $1 \not \in K_0$.
\item[(FR5)] This is immediate from the definition of $-1_R$ and Lemma
  \ref{units-are-singletons}. Indeed, any $T \in K^\times$ should be a singleton, say $T=\{a\}$. Hence
  $(1+T) \in K_0$ if and only if $0_R \in \{1_R\} + a = 1_R + a$, and this last equality is
  equivalent to the statement $a=-1_R$ (using the hyperring notion of additive inverse).

\item[(FR6)] Let $A_1,A_2,B_1,B_2 \in K$ such that $A_i+B_i \in K_0$ for $i=1,2$. This implies that
  we have $a \in A_1, -a \in B_1$ and $b \in A_2, -b \in B_2$ for some $a,b \in R$. Therefore,
  $ab \in (A_1\times A_2)\cap (B_1\times B_2)$. Since $ \varepsilon =\{-1_R\}$, we have
  $0_R \in ab - ab \in A_1\times A_2 +\varepsilon B_1\times B_2$. Therefore,
  $A_1\times A_2 +\varepsilon B_1\times B_2 \in K_0$.
\item[(FR7)] Let $A,B,C,D \in K$ such that $A+B\times(C+D) \in K_0$. This means that we have
  $a \in A,b \in B, c\in C,d\in D$ such that $0_R \in a+b(c+d)=a+bc +bd$. However,
  $bc \in B\times C$ and $bd \in B\times D$. Therefore $A+(B\times C) + (B\times D ) \in
  K_0$. \qedhere
\end{enumerate}
\end{proof}

\begin{example}\label{krasner}
  Let $\mathbb{K}=\{0,1\}$ be the Krasner's hyperfield. Under the construction of Theorem
  \ref{object}, one obtains $K=\{0,1, \mathbb{K}\}$. Here one can check that the addition and
  multiplication on $K$ is same as the fuzzy ring in Example \ref{krasnerfuzzy} by letting
  $\mathbb{K}=k_0$. Furthermore, as mentioned above, when one restricts the addition and
  multiplication to $\{0,1\} \subseteq K$, they agree with the operations in the Krasner
  hyperfield $\mathbb{K}$.
\end{example}

\begin{example}\label{signsexample}
  Let $\mathbb{S}=\{0,1,-1\}$ be the hyperfield of signs. From the construction of Theorem
  \ref{object}, one obtains $K=\{0,1,-1,\{0,1\},\{1,-1\},\{0,-1\},\mathbb{S}\}$. Let us restrict $K$
  to $K'=\{0,1,-1,\mathbb{S}\}$. Then one can check that the addition and multiplication on $K'$ is
  same as Example \ref{signfuzzy}. Furthermore, as in Example \ref{krasner}, if we restrict the
  operations on $K'$ to $\{0,1,-1\}$, then they are exactly same as in $\mathbb{S}$.
\end{example}

\begin{remark}
 We will see that the previous examples are particular cases of the `reduced' construction from $\S \ref{demi}$ applied to $\mathbb{K}$ and $\mathbb{S}$, respectively
\end{remark}

Let $f:R_1\to R_2$ be a morphism of hyperrings and let $\tilde R_i=\functor(R_i)$ be the associated fuzzy rings $\{\mathcal{P}^*(R_i),+,\times;-1,R^0_i\}$, where $\mathcal{P}^*(R_i)$ is the set of nonempty subsets of $R_i$, the addition and the multiplication are as in Theorem \ref{object}, and $R^0_i=\{A \in \mathcal{P}^*(R_i) \mid 0_{R_i} \in A\}$. We define a map
 \begin{equation}\label{morphismdef}
 \functor{(f)}:\tilde{R}_1 \longrightarrow \tilde{R}_2, \quad A \mapsto f(A), 
 \end{equation}
 where $f(A):=\{f(a) \mid a \in A\subseteq R_1\}$. 

\begin{proposition}\label{homomor}
 For every morphism $f:R_1\to R_2$ of hyperrings, the associated map $\functor(f):\tilde R_1\to\tilde R_2$ is a strong morphism of fuzzy rings.
\end{proposition}
\begin{proof}
 For notational convenience, we identify singletons with
 elements and let $\functor(f):=\tilde{f}$. We have to check the
 following:
\begin{equation}\label{condition1}
  \tilde{f}(1)=1, \quad \tilde{f}(0)=0, \quad \tilde{f}(a\times b)=\tilde{f}(a)\times\tilde{f}(b)\textrm{ for }a \in \tilde{R}_1^\times, b \in \tilde{R}_1
\end{equation}
\begin{equation}\label{condition2}
  \sum_{i=1}^n\tilde{f}(a_i)\times \tilde{f}(b_i) \in \tilde{R}_2^0 \textrm{ whenever } \sum_{i=1}^na_i\times b_i \in \tilde{R}^0_1\textrm{ for } a_i,b_i \in \tilde{R}_1.
\end{equation}
Now clearly, $\tilde{f}(1)=1$ and $\tilde{f}(0)=0$ since $1_{\tilde{R}_i}=\{1_{R_i}\}$ and
$0_{\tilde{R}_i}=\{0_{R_i}\}$. Also, $\tilde{f}(a\times b)=\tilde{f}(a)\times\tilde{f}(b)$ holds
since $f$ is a homomorphism of hyperrings. All it remains to show is the last condition
\eqref{condition2}. Suppose that $\sum_{i=1}^na_i\times b_i \in \tilde{R}^0_1$. This implies that
there exist $x_i \in a_i$ and $y_i \in b_i$ ( $x_i,y_i \in R_1$) such that
$0 \in \sum_{i=1}^nx_i\times y_i$. Therefore, we have $0 \in \sum_{i=1}^nf(x_i)\times f(y_i)$ since
$f$ is a homomorphism of hyperrings. However, since $f(x_i) \in \tilde{f}(a_i)$ and
$f(y_i) \in \tilde{f}(b_i)$, we have
$\sum_{i=1}^n\tilde{f}(a_i)\times \tilde{f}(b_i) \in \tilde{R}_2^0$. This proves that $\tilde{f}$ is
a strong morphism.
\end{proof}

It is straightforward to see that $f\mapsto \functor(f)$ respects compositions and identity
morphisms, and so it is a functor.  Extending the fact that the category of ordinary rings faithfully embeds into
fuzzy rings with strong morphisms, we have the following result.

\begin{proposition}
  The functor $\functor:\hyperrings \to \fuzzystrong$ is faithful.
\end{proposition}
\begin{proof}
  Suppose that $f,g: R_1 \longrightarrow R_2$ are two homomorphisms of hyperrings such that
  $\functor(f)=\functor(g)$. Then, in particular, $\functor(f)$ and $\functor(g)$ agree on
  singletons in $R_1$, and this implies that $f = g$.
\end{proof}

We note that the functor $\functor$ is not full in general, even restricted to the subcategory of
hyperfields, as demonstrated by the following example.

\begin{example}\label{notfullexample}
  Let $\mathbb{K}=\{0,1\}$ be the Krasner hyperfield and consider the hyperfield
  $\mathbb{Q}(x) /\mathbb{Q}^\times$ obtained from the field of rational functions $\mathbb{Q}(x)$
  as a quotient hyperring via Construction \ref{quotientconstruction}.  Since any morphism of
  hyperfields must send $0$ to $0$ and $1$ to $1$, there can be at most one morphism
  $f:\mathbb{K} \to \mathbb{Q}(x) /\mathbb{Q}^\times$, and in fact there is exactly one since in
  $\mathbb{Q}(x) /\mathbb{Q}^\times$ we have $1+1= \{0,1\}$. In contrast, we claim that in the
  category of fuzzy rings and strong morphisms, there is more than one strong morphism
  \[
\functor(\mathbb{K}) \to \functor(\mathbb{Q}(x) /\mathbb{Q}^\times). 
\]
We have the morphism $\functor(f)$, and we will now construct a second morphism.  Consider the
mapping
  \[
  \varphi: \functor(\mathbb{K})=\{\{0\},\{1\},\mathbb{K}\} \longrightarrow \functor(\mathbb{Q}(x)
  /\mathbb{Q}^\times),
\]
\begin{align*}
0 & \mapsto 0,\\
1 &\mapsto 1, \\
\mathbb{K} & \mapsto \mathbb{Q}(x)/\mathbb{Q}^\times.
  \end{align*}
  As maps of sets, $\varphi$ and $\functor(f)$ are distinct since
  $\functor(f)(\mathbb{K}) = \{0,1\}$ and $\varphi(\mathbb{K}) = \mathbb{Q}(x) /\mathbb{Q}^\times$,
  which has more elements than just 0 and 1.  It remains to verify that $\varphi$ is indeed a strong
  morphism. Since $\functor(\mathbb{K})^\times=\{1\}$, it is immediate that $\varphi$ satisfies
  condition (1) of Definition \ref{def:strong-morphism}.  Now, suppose that $\sum_{i=1}^n a_ib_i$ is
  null in $\functor(\mathbb{K})$.  If all summands are equal to $\{0\}$ then it
  must be the case that $a_i$ or $b_i$ is $\{0\}$ for all $i$, and so $\sum_{i=1}^n \varphi(a_i)\varphi(b_i)$
  is $\{0\}$ and thus trivially null.  If at least one of the summands $a_j b_j$ is not null, then
  either there is at least one summand $a_j b_j$ equal to $\mathbb{K}$ or there are at least two
  summands ($a_j b_j$ and $a_k b_k$) equal to $\{1\}$. In the first case, one of $a_j$ or $b_j$ is
  $\mathbb{K}$ and the other is $\{1\}$ or $\mathbb{K}$, and so
  $\varphi(a_j)\varphi(b_j)=\mathbb{Q}(x) /\mathbb{Q}^\times$ and therefore the sum
  $\sum_{i=1}^n \varphi(a_i)\varphi(b_i)$ is null.  In the second case, $a_j=b_j=1$ so
  $\varphi(a_j)\varphi(b_j)=\{1\}$ and the same for the $k$ terms, and thus in the summation
  $\sum_{i=1}^n \varphi(a_i)\varphi(b_i)$, the element $\{1\}$ occurs at least twice and therefore the sum null.
\end{example}

Passing from strong to weak morphisms eliminates the phenomenon in the above example and we have:

\begin{proposition}
  Restricting $\functor$ to hyperfields, the composition
  \[\hyperfields \stackrel{\functor}{\to} \fuzzystrong \to \fuzzyweak\] is
  fully faithful.
\end{proposition}
\begin{proof}
  Let $f:R_1\longrightarrow R_2$ be a homomorphism of hyperfields. Since the group of multiplicative
  units of $R_i$ is canonically isomorphic to the group of units of the fuzzy ring $\functor(R_i)$,
  and since $R_i$ are both hyperfields (so all nonzero elements are units), it follows that
  $\functor$, when regarded as a functor $\hyperfields \to \fuzzyweak$, is faithful.

  Next, we show that it is full.  Suppose that
  $\varphi: \functor(R_1)^\times \longrightarrow \functor(R_2)^\times$ is a weak morphism.  Again,
  since $\functor(R_i)^\times \cong R_i^\times$, this induces a map of sets $f: R_1 \to R_2$ by the
  rule $\{f(a)\}=\varphi(\{a\})$ for $a\neq 0$, and defining $f(0)=0$.  By construction $f$ is
  multiplicative, so we only have to verify the additivity condition:
  \[ 
  f(a+b) \subseteq f(a)+f(b), \quad \forall a,b \in R_1.
  \]
  Given an element $c \in a+b$, we have
  \begin{equation}\label{con}
    0 \in c+(-c) \subseteq (a+b)+(-c)
  \end{equation}
  Note that, since $\varphi({-a})=-\varphi({a})$, we have $f(-a)=-f(a)$ for all $a \in R_1$. Hence
  the containment \eqref{con} implies that $\{a\}+\{b\}+\{-c\}$ is null in the fuzzy ring
  $\functor(R_1)$, and hence $\varphi(\{a\})+\varphi(\{b\})+\varphi(\{-c\})$ is null in
  $\functor(R_2)$.  In other words, $0 \in f(a)+f(b) - f(c)$, so $f(c) \in f(a)+f(b)$ and hence $f$
  is a homomorphism of hyperfields and $\varphi = \functor(f)$ This proves that the functor
  $\functor: \hyperfields \to \fuzzyweak$ is full.
\end{proof}

The following examples show that the restriction from hyperrings to hyperfields in the above
proposition is necessary for both fullness and faithfulness. 

\begin{example} \label{notfullhyperrings} This example illustrates that fullness requires
  restricting to hyperfields.  For any (multiplicative) abelian group $H$ of order at least $4$, one
  can canonically associate a hyperfield $\mathbb{K}[H]$. (We refer the readers to \cite[\S 3]{con3}
  and the references therein for more details.)  Briefly, the underlying set of $\mathbb{K}[H]$ is
  given by $H \cup \{0\}$. Multiplication is that of $H$ with $0$ as an absorbing element, and
  addition is given by the following rule.
  \begin{align*}
    a+ 0 & = a,\\
    a + a & = \{0,a\},\\
  \end{align*}
  and if $a$ and $b$ are distinct nonzero elements then $a+b=H \smallsetminus \{a,b\}$. Moreover, we
  can adjoin two multiplicatively idempotent elements $e$ and $f$ to obtain another hyperring
  $\mathbb{K}[H]\cup \{e,f\}$ with the presentation:
  \begin{equation} \label{presentation} e^2=e,\quad f^2=f, \quad ef=0,\quad ah=a, \quad b+b=\{0,b\}, \quad b+c=H\cup \{e,f\} \smallsetminus \{b,c\} 
  \end{equation}
  where $h\in H$, $a\in\{e,f\}$ and $b,c\in H\cup\{e,f\}$ with $b\neq c$.
  Let $R:=\mathbb{K}[H]\cup\{e,f\}$ and $L:=\mathbb{K}[H]$.  We will construct a weak fuzzy ring
  morphism $\functor(R) \to \functor(L)$ that does not come from a hyperring homomorphism.
  
  First observe that $R^\times \cong L^\times \cong H$. Therefore the identity map
  $i:R^\times \longrightarrow L^\times$ defines a weak morphism from the fuzzy ring $\functor(R)$ to
  $\functor(L)$ since $\functor(R)^\times \cong R^\times$ and $\functor(L)^\times \cong L^\times$.
  We will prove by contradiction that this weak morphism does not come from a hyperring
  homomorphism.

  Suppose that $g:R \longrightarrow L$ is a homomorphism of hyperrings inducing the weak morphism
  $i$, so $g|_{R^\times}$ is the identity on $H$. Since $ef=0$ and $L$ is a hyperfield, either
  $g(e)=0$ or $g(f)=0$. In fact, both $g(e)$ and $g(f)$ must be zero; for if $g(e)=u$ for some $u \in H$
  then it follows from the presentation \eqref{presentation} that for any $h \in H$ we have
  \[
  g(eh)=g(e)=g(e)g(h)=g(e)h
  \]
  since $g$ is a homomorphism of hyperrings, $h \in H=R^\times$, and $g|_{R^\times}=i$, and this
  contradicts the hypothesis that $H$ is not the trivial group.  Therefore $g(e)=g(f)=0$.

  Next, from the definition, we have $1 \in e+f$ and hence $0 \in 1+e+f$. 
  Since $g$ is a homomorphism of hyperrings, we should have
  $g(0)=0 \in g(1)+g(e)+g(f)$. But $g(1)+g(e)+g(f)=1+0+0=1$ and this gives the desired
  contradiction.
 \end{example}
 
The preceding example also shows that the canonical functor $\fuzzystrong \to \fuzzyweak$
is not full, as explained in the following example.

\begin{example}\label{notfull}
  Using the same notation as in Example \ref{notfullhyperrings}, we have the weak morphism
  $i:\functor(R) \longrightarrow \functor(L)$, and we claim that it is not the restriction of a
  strong morphism $g:\functor(R) \longrightarrow \functor(L)$.  The proof is by
  contradiction. Suppose there is such a strong morphism $g$.  Then $g(e)=g(f)=0$ since
  $g(eu)=g(e)g(u)$ from the condition \eqref{modulecondition} of a strong morphism and the fact that
  $u \in R^\times$.  On the other hand, $0 \in e+f+1$, which means that $(e+f+1)$ is null.  However
  $1=g(e)+g(f)+g(1)$ and hence $(g(e)+g(f)+g(1))$ is not null, which contradicts the hypothesis that
  $g$ is a strong morphism.
\end{example}


\section{Characterizing the essential image of the functor
  \texorpdfstring{$\functor: \hyperfields \to \fuzzyweak$}{from hyperfields to fuzzy
    rings}} \label{invfunctor} In this section, we determine the essential image of the fully
faithful functor $\functor:\hyperfields \to \fuzzyweak$ as the full subcategory of $\fuzzyweak$
whose objects are fuzzy rings satisfying the following assumption.

We say that a fuzzy ring $(K,+,\times,\varepsilon,K_0)$ is \emph{field-like} if for each pair of
units $a,b\in K^\times$, there exists an element $c\in K^\times\cup\{0\}$ such that $a+b+c\in K_0$.
We denote the full subcategory of $\fuzzyweak$ whose objects are field-like fuzzy rings by
$\fuzzyweakast$.  

The condition of being field-like is not vacuous, as this example shows.
\begin{example}
Consider the ring of integers $\mathbb{Z}$ as a hyperring, and the corresponding fuzzy ring
$K= \functor(\mathbb{Z})$. The underlying set is all nonempty subsets of the integers, the null elements are
the nonempty sets that contain zero, and the units are the singletons $\{1\}$ and $\{-1\}$.  Taking
$a=b=\{1\}$, there is no unit $c$ such that $a+b+c$ is null since $\{1\}+\{1\}+\{1\}=\{3\}$ and
$\{1\}+\{1\}+\{-1\}=\{1\}$.  Thus $\functor(\mathbb{Z})$ is not field-like.
\end{example}

Note that for every hyperfield $k$, the fuzzy ring $\functor(k)$ is field-like, as can be seen as
follows. For all $a,b\in k$, the hypersum $a + b$ contains an element $c'\in k$. Thus $0\in a+b-c'$
and so in $\functor(k)$ we have $a+b+c\in \functor(k)_0$ for $c=-c'$. Thus the essential image of
the functor $\functor:\hyperfields \to \fuzzyweak$ is contained in $\fuzzyweakast$.

We may thus regard $\functor$ as having codomain $\fuzzyweakast$.  We have already seen that it is
fully faithful, and we now show that it is essentially surjective onto $\fuzzyweakast$ by
constructing a quasi-inverse.

Given a field-like fuzzy ring $(K,+,\times,\varepsilon,K_0)$, we let $\invfunctor(K)$ be the
commutative multiplicative monoid $\{0\}\cup K^\times$ equipped with the following hyperoperation
$\oplus$: for $a,b \in \invfunctor(K)$,
\begin{equation}
a\oplus b :=\{c \in \invfunctor{(K)} \mid a+b+\varepsilon c \in K_0\}.
\end{equation}

\begin{lemma}
With the above construction, $\invfunctor(K)$ is a hyperfield. 
\end{lemma} 

\begin{proof}
  Since $\invfunctor(K)$ is a commutative (multiplicative) monoid, we only have to check that
  $a\oplus b$ is non-empty, that $(\invfunctor(K),\oplus)$ is a canonical hypergroup and that
  $\oplus$ and $\times$ are compatible.

  That the set $a\oplus b$ is non-empty can be seen as follows. Since $K$ is field-like, there
  exists an element $c\in K^\times\cup\{0\}$ such that $a+b+c\in K_0$. Since multiplication by $\epsilon$
  leaves $K^\times\cup\{0\}$ invariant, $\epsilon c\in \invfunctor(K)$, and since $\epsilon^2=1$, we
  conclude that $\epsilon c\in a\oplus b$.

  We continue with the proof that $(\invfunctor(K),\oplus)$ is a canonical hypergroup. Note that
  property (5) of Definition \ref{canonicalhypergroup} is implied by the other axioms of a
  hyperring. Property (1) (commutativity) follows from (FR0).

  As a preparation, we note that $0+c\in K_0$ if and only if $c=0$. Indeed, if $c=0$, then
  $0+c=0\in K_0$. If $c\neq 0$, then (FR5) implies that $0=\varepsilon c$, which is impossible.

  Property (2) (identity) requires a separation of cases. For $a=0$ and $c\in \invfunctor(K)$, we
  have $a+0+\varepsilon c\in K_0$ if and only if $c=0$, by the previous observation. Thus
  $0\oplus 0=\{0\}$. For $a,c\in \invfunctor(K)$ with $a\neq0$, we have $a+0+\varepsilon c\in K_0$
  if and only if $\varepsilon c=\varepsilon a$, by (FR5). This is equivalent to $c=a$ since
  $\varepsilon^2=1$ by (FR3). Thus $a\oplus 0=\{a\}$ for $a\neq0$.

  Property (3) (inverses) also requires a separation of cases. Since $K_0$ and $K^\times$ are
  disjoint, we have $0\in 0\oplus c$, i.e., $0+c+\varepsilon 0\in K_0$, if and only if $c=0$. For a
  nonzero $a\in \invfunctor(K)$, we have $0\in a\oplus c$, i.e., $a+c=a+c+\varepsilon 0\in K_0$, if
  and only if $c=\varepsilon a$ by (FR5).

  Property (4) (associativity) follows from the following:
  \[(a\oplus b)\oplus c=\{d\in \invfunctor(K) \mid a+b+c+\epsilon d\in K_0\}=a\oplus (b\oplus c),\]
  where we leave the details to the reader.

  Compatibility of $\times$ and $\oplus$, i.e., $a\times (b\oplus c)=a\times b\oplus a\times c$,
  follows for nonzero $a$ from (FR2) and is obvious for $a=0$.
\end{proof}

Let $K$ and $L$ be field-like fuzzy rings and $f:K\to L$ a weak morphism.  We extend $f$ from a map
$K^\times \to L^\times$ to a map
$\invfunctor(f): \invfunctor(K) = K^\times \cup \{0\} \to \invfunctor(L) = L^\times \cup \{0\}$ by
sending 0 to 0.

\begin{lemma}
The map $\invfunctor(f):\invfunctor(K)\to \invfunctor(L)$ is a morphism of hyperfields. 
\end{lemma}

\begin{proof}
  For the notational convenience, we let $g :=\invfunctor(F)$, $k :=\invfunctor(K)$, and
  $\ell :=\invfunctor(L)$. Then clearly, $g(0)=0$, $g(1)=1$ and $g(ab)=g(a)g(b)$. We continue with
  verifying additivity. Since $0$ is the neutral element for addition, we have that
  $\sum a_i\in K_0$ implies $\sum f(a_i)\in L_0$ for $a_i\in K^\times\cup\{0\}$ if we define
  $f(0)=0$. This allows us to avoid a separation of cases for zero elements. We conclude that if
  $c\in a \oplus b$ for $a,b,c\in k$, then $a+b+\epsilon c\in K_0$ and thus
  $f(a)+f(b)+\epsilon f(c)\in L_0$ and thus $g(c)\in g(a)\oplus g(b)$, as desired.
\end{proof}

Clearly, the construction $\invfunctor$ sends the identity map to the identity map and respects
compositions of morphisms. Thus it yields a functor $\invfunctor:\fuzzyweakast\to \hyperfields$.  It
is immediate from the constructions of $\functor$ and $\invfunctor$ that any hyperfield $k$ is
canonically isomorphic to $\invfunctor\circ\functor(k)$.  By this discussion, we have:

\begin{lemma}\label{lem:iso1}
There is a natural isomorphism between
$\invfunctor \circ \functor$ and the identity functor on $\hyperfields$
\end{lemma}

We next show that there is also a natural isomorphism between the identity of $\fuzzyweakast$ and
$\functor \circ \invfunctor$.  By Lemma \ref{units-are-singletons} and the definition of
$\invfunctor$, for any fuzzy ring $K$ there are canonical isomorphisms of multiplicative monoids
$K^\times \cong \invfunctor(K)^\times \cong (\functor\circ \invfunctor(K))^\times$.  

\begin{lemma}\label{lem:iso2}
The isomorphism of multiplicative groups $K^\times \cong (\functor\circ \invfunctor(K))^\times$ is
a weak isomorphism.
\end{lemma}
\begin{proof}
  Let $L$ denote the fuzzy ring $\functor\circ \invfunctor(K)$, and write $\alpha$ for the group
  isomorphism $K^\times \cong L^\times$.  Let $a_i$ be elements of $K^\times$. By the construction
  of $\functor\circ\invfunctor(K)$, we have that $\sum \alpha(a_i)$ is null if and only if
  $0\in \sum a_i$ holds in the hyperfield $\invfunctor(K)$, and this in turn is the case if and only
  if $\sum a_i$ is null in $K$. This shows that both $\alpha$ and $\alpha^{-1}$ are weak
  morphisms of fuzzy rings, and hence they are weak isomorphisms.
\end{proof}

\begin{theorem}\label{thm: equivalence of hyperfields and weak fuzzy rings}
The functor $\functor:\hyperfields\to\fuzzyweakast$ is an equivalence of categories.
\end{theorem}

\begin{proof}
  It follows immediately from Lemmas \ref{lem:iso1} and \ref{lem:iso2} that the $\functor$
  and $\invfunctor$ are mutual inverses.
\end{proof}

\begin{remark}
  Let $R$ be a hyperring. Then $\functor(R)$ is a field-like fuzzy ring if and only if
  $R^\times\cup\{0\}$ is a sub-hyperfield of $R$, which means that for all $a,b\in R^\times$, the
  intersection of $a+b$ with $R^\times\cup\{0\}$ is non-empty. In this case,
  $\invfunctor(\functor(R))$ is canonically isomorphic to the sub-hyperfield $R^\times\cup\{0\}$.
\end{remark}

\subsection{Extending the equivalence beyond field-like fuzzy rings}\label{subsection: Extension to partial hyperfields}
The equivalence of categories, given in Theorem \ref{thm: equivalence of
  hyperfields and weak fuzzy rings}, between hyperfields and field-like fuzzy rings admits an
extension to an equivalence between $\fuzzyweak$ and a category of objects very mildly generalizing
hyperfields.  

We define a \emph{partial hyperring} to be a set $R$ equipped with an ordinary binary operation
$\times$ and a set-valued operation $+$, such that these operations satisfy all of the axioms of a
hyperring except that sums $a+b$ are not required to be empty.  A morphism between partial
hyperrings $R$ and $S$ is a multiplicative map $f:R\to S$ with $f(0_R)=0_S$ and $f(1_R)=1_S$ such
that $f(a+_Rb)\subset a+_Sb$.  This defines the category $\parthyperrings$ of partial
hyperrings. The full subcategory $\parthyperfields$ of partial hyperfields consists of all partial
hyperrings $R$ such that every nonzero element is a admits a multiplicative inverse.
  
The construction of the functor $\functor$ extends literally to a functor
$\functor:\parthyperrings\to\fuzzystrong$. Restricting the domain to partial hyperfields and
restricting the codomain to weak morphism of fuzzy rings yields a functor
\[\functor' :\parthyperfields\to\fuzzyweak,\]
which turns out to be an equivalence of categories.  Indeed, it is straightforward to check that the
functor $\invfunctor:\fuzzyweakast\to\hyperfields$ extends to a functor
$\invfunctor' :\fuzzyweak\to\parthyperfields$, inverse to $\functor'$, by the very same construction
as above, with the only difference being that a fuzzy ring $K$ that is not field-like gives rise to
a partial hyperfield $\invfunctor(K)$ that contains elements whose hypersum is empty.

We draw the attention to the following effect: every partial hyperring $R$, and in particular every
hyperring, contains a canonical partial hyperfield, which is $R^\times\cup\{0\}$, which we call the
\emph{unit field of $R$}. For example, the unit field of the integers $\mathbb Z$ is the partial
hyperfield $\{0,\pm 1\}$ for which all sums are defined as in $\mathbb Z$, with exception of $1+1$
and $(-1)+(-1)$, which are empty.

\section{Totally ordered abelian groups, hyperfields, and fuzzy rings}\label{review}

In this section we review two constructions that take a totally ordered abelian group as input.  The
first produces a hyperfield and the second produces a fuzzy ring as output.  We describe the
relation between these two constructions.

\subsection{From totally ordered abelian groups to hyperfields} \label{linearorderedhypergroup}


Let $(\Gamma,\times)$ be a totally ordered abelian group. Viro \cite{viro} observed that, from this,
one can construct a hyperfield $H_\Gamma$ as follows. First adjoin an element $0$ and extend
the ordering to $\Gamma\cup \{0\}$ by $0 < a$ for all $a\in \Gamma$.  The hyperaddition
is
\begin{equation}
  x \oplus y =\left\{ \begin{array}{ll}
                       \max\{x,y\} & \textrm{if $x\neq y$}\\
                       \left[0,x\right]& \textrm{if $x=y$}
\end{array} \right.
\end{equation}
and the multiplication is given by the group multiplication $\times$ of $\Gamma$, with $0$ as
an absorbing element.

When $\Gamma$ is the group $(\mathbb{R},+)$ then the hyperfield $H_\Gamma$ is the tropical
hyperfield denoted $\mathbb{T}$ and the element $0$ is called $-\infty$.

The following has been proven by Viro.

\begin{theorem}[{\cite[$\S 4.7$]{viro}}]
  Let $\Gamma$ be a totally ordered abelian group. Then the hyperfield $H_\Gamma$ is
  doubly-distributive.
\end{theorem}

\subsection{From totally ordered abelian groups to fuzzy rings}

There is a fuzzy ring analogue of the construction $\Gamma \mapsto H_\Gamma$ (see
\cite{dress1992valuated} and \cite{dress2011algebraic}), which we now review.  The input, once
again, is a totally ordered abelian group $(\Gamma, \times)$ and the output is a fuzzy ring.

Let $K_\Gamma\subset \mathcal{P}(\Gamma\cup \{0\})$ be the set consisting of all singletons
and all intervals $[0, a]$.  We define a multiplication operation $\boxtimes$ on $K_\Gamma$
induced from the multiplication $\times$ of $\Gamma$, extended to $\Gamma \cup \{0\}$ with
$0$ as an absorbing element.  The addition operation $\boxplus$ is defined as follows: for
$A,B \in K_\Gamma$,
\begin{equation}
  A\boxplus B =\left\{ \begin{array}{lllll}
                         \max\{a,b\} & \textrm{if $A=\{a\}$, $B=\{b\}$}\\
                         \left[0,b\right]& \textrm{if $A=\{a\}$, $B=[0,b]$, and $a\leq b$}\\
                         a & \textrm{if $A=\{a\}$, $B=[0,b]$, and $b < a$}\\
                         A & \textrm{if $A=[0,a]$, $B=[0,b]$, and $b < a$}\\
                         B & \textrm{if $A=[0,a]$, $B=[0,b]$, and $a\leq b$}
\end{array} \right.
\end{equation}
The set of null elements $(K_\Gamma)_0$ consists of those subsets that contain $0$, and the
element $\epsilon$ is the singleton consisting of the identity element of $\Gamma$.  Then
$K_\Gamma = (K_\Gamma;\boxplus; \boxtimes; \epsilon; (K_\Gamma)_0)$ is a fuzzy ring.  

\begin{remark}
  One may notice that the above construction of Dress and Wenzel is similar to the functor
  $\functor$ which we introduced in $\S \ref{functor}$. We will use the fact that $H_{\Gamma}$
  is doubly-distributive in $\S \ref{demi}$ to generalize the construction $K_\Gamma$ of Dress and
  Wenzel by constructing a faithful functor from the category of doubly-distributive hyperfields to
  the category of fuzzy rings.
\end{remark}

\subsection{Zariski systems}

A notion of Zariski systems for fuzzy rings was introduced in \cite{dress2011algebraic}. In this
subsection, we show that for a totally ordered abelian group $\Gamma$, there is a canonical
injective homomorphism of fuzzy rings from $K_\Gamma$ to $\functor(H_\Gamma)$. Furthermore, this
injection pushes forwards a Zariski system on $K_\Gamma$ to a Zariski system in $\functor(H_\Gamma)$
and the sets defined by Zariski systems coincide. For the notational convenience, we let
$\functor(H_\Gamma):=\widetilde{H_\Gamma}$.

\begin{proposition}\label{inj}
  Let $\Gamma$ be a linearly ordered abelian group and $H_\Gamma$ be the hyperfield as above. Then
  the fuzzy ring $K_{\Gamma}$ is the sub-fuzzy ring of $\widetilde{H_\Gamma}$.
\end{proposition}
\begin{proof}
  This is straightforward. For the notational convenience, let $\widetilde{H_{\Gamma}}:=H$. Define
  the following map:
  \[
  \mathit{i}:K_\Gamma \longrightarrow H, \quad A \mapsto A.
  \]
  By the definition of $\mathit{i}$ and $H$, $\mathit{i}$ is clearly injective. We also have
  $\mathit{i}(e)=\{e\}$, where $e$ is the identity of $\Gamma$ and
  $\mathit{i}(0)=\{0\}$. One can easily check that since the definition of addition and
  multiplication is exactly same as $H$, we have $\mathit{i}(ab)=\mathit{i}(a)\mathit{i}(b)$
  $\forall a,b \in K_{\Gamma}$. Finally, suppose that
  $\sum_{k=1}^na_k\times b_k \in (K_\Gamma)_0=K_I \cup \{0\}$. This implies that all
  $a_k \times b_k =0$ or there exists $m \neq r$ such that
  $a_k\times b_k\leq a_m\times b_m =a_r \times b_r$ for $k \neq m,r$. In the first case, we clearly
  have $\sum_{k=1}^n\mathit{i}(a_k)\times \mathit{i}(b_k) \in H_0$. In the second case, we have
  $\mathit{i}(a_k)\times \mathit{i}(b_k)\leq \mathit{i}(a_m)\times \mathit{i}(b_m) =\mathit{i}(a_r)
  \times \mathit{i}(b_r)$
  for $k \neq m,r$. Therefore, we have $\sum_{k=1}^n\mathit{i}(a_k)\times \mathit{i}(b_k) \in
  H_0$. This proves that $\mathit{i}$ is a strong morphism of fuzzy rings.
\end{proof}

In what follows we use the same notations and terms as in \cite{dress2011algebraic}. To recall, a Zariski system with coefficients in a fuzzy ring $K$ is a triple $(M,K,\mathcal{F})$ where $M$ is a set and $\mathcal F$ is a multiplicatively closed set of functions from $M$ to $K$ that contains for every $a\in M$ a function $f$ such that $f(a)\notin K_0$. Heuristically, one should think of $M$ as a variety over $K$ and of $\mathcal{F}$ as its set of regular functions.

Let $\mathcal{S}:=(M,K_{\Gamma},\mathcal{F})$ be a Zariski system. Since $\mathcal{F}$ is a subset
of $K_{\Gamma}^M$ (a set of maps from $M$ into $K_{\Gamma}$), by the injection $i$ in Proposition
\ref{inj}, we have the following map:
\begin{equation}
  \varphi:K_{\Gamma}^M \longrightarrow \widetilde{H_{\Gamma}}^M, \quad f \mapsto i\circ f:=\tilde{f}.
\end{equation}
Let $\mathcal{F}':=\varphi(\mathcal{F})$. Then we have the following.

\begin{lemma}\label{zariskisystem}
  $S':=(M,\widetilde{H_{\Gamma}}, \mathcal{F}')$ is a Zariski system.
\end{lemma}
\begin{proof}
  Let $H:=\widetilde{H_{\Gamma}}$.
  \begin{itemize}
  \item[(Z1)] Suppose that $\tilde{f},\tilde{g} \in \mathcal{F}'$. For any $a \in M$, since
    $\mathit{i}$ is a strong morphism, we have
    \[
    \tilde{f}(a)\times \tilde{g}(a)=\mathit{i}(f(a))\times \mathit{i}(g(a))=\mathit{i}(f(a)\times
    g(a))=\mathit{i}((f\times g)(a))=\mathit{i}\circ (f\times g)(a).
    \]
    But, since $\mathcal{S}$ is a Zariski system, $f \times g \in \mathcal{S}$ and hence
    $\tilde{f}\times \tilde{g} \in \mathcal{F}'$.
  \item[(Z2)] For $a \in M$, we have $f \in \mathcal{M}$ such that
    $f(a) \in K_{\Gamma} \backslash (K_\Gamma)_0$. Then one can easily check that
    $\tilde{f}(a) \in H \backslash H_0$.\qedhere
\end{itemize}
\end{proof}

The following proposition shows that $\mathcal{S}$ and $\mathcal{S}'$ define the same `solution set'. 

\begin{proposition}
  Let $\mathcal{S}:=(M,K_{\Gamma},\mathcal{F})$ be a Zariski system and $\mathcal{T}$ be a subset of
  $\mathcal{F}$. Let $S'$ be the Zariski system as in Lemma \ref{zariskisystem}. Then
  \[
  Z(\mathcal{T})=Z(\varphi(\mathcal{T})).
  \]
\end{proposition}
\begin{proof}
  For notational convenience, we write $K=K_\Gamma$ and $H=\widetilde{H_{\Gamma}}$. Suppose that
  $a \in Z(\mathcal{T})$. This means that $f(a) \in K_0$ for all $f \in \mathcal{F}$, hence
  $0 \in f(a)$ and $0 \in \tilde{f}(a)$. It follows that
  $a \in Z(\varphi(\mathcal{T}))$. Conversely, suppose that $a \in Z(\varphi(\mathcal{T}))$. This
  implies that $\mathit{i}(f(a)) \in H_0$ for all $f \in \mathcal{F}$. But, this happens only if
  $0 \in f(a)$ and hence $a \in Z(\mathcal{T})$.
\end{proof}


\section{From doubly-distributive hyperfields to partial demifields and fuzzy rings} \label{demi}

In this section we restrict our attention to the subcategory $\ddhyperfields$ of doubly-distributive
hyperfields with strict morphisms. We construct the reduced variant
\[
\reducedfunctor:  \ddhyperfields \to \fuzzystrong
\]
of the functor $\functor$, along with the factorization
$\reducedfunctor = \demifunctorB \circ \demifunctorA$ through a subcategory of partial demifields
asserted in Theorem \ref{thm:dd-factorization}.

\begin{remark}
We remark that the four main examples in \cite{baker2016matroids} are all doubly-distributive. 
\end{remark}

\subsection{The construction of \texorpdfstring{$\reducedfunctor$}{F-bar}}

Let $F$ be a hyperfield, and let $S(F)\subset \mathcal{P}(F)^*$ denote the set of all non-empty subsets of
$F$ formed by taking hyperaddition sums of finitely many elements of $F$, i.e.,
\[
 S(F) := \left\{A \subseteq F \mid A = \sum_{i=1}^n a_i, \:\: a_i\in F,\: n\in\mathbb{N}\right\}.
\]
Recall that we defined binary operations $+$ and $\times$ on
$\mathcal{P}(F)^*$ by the formulae
\begin{align*}
A + B & :=  \bigcup_{a\in A, \: b \in B} a +_F b, \\
A \times B & = \{a \times_F b \mid a\in A, \: b\in B\}.
\end{align*}

\begin{lemma}
Let $F$ be a hyperfield.  
\begin{enumerate}
\item The binary operation $+$ on $\mathcal{P}(F)^*$ restricts to a binary operation on $S(F)$.
\item If $F$ is doubly-distributive then the binary operation $\times$ on $\mathcal{P}(F)^*$ restricts to a
  binary operation on $S(F)$, and $(S(F),+, \times )$ is a semiring.
\item If $f:F \longrightarrow F'$ is a strict homomorphism of doubly-distributive hyperfields then
  the induced map $S(F) \to S(F')$ is a semiring homomorphism.
\end{enumerate}
\end{lemma}
\begin{proof}
  Suppose $A,B\in S(F)$, so $A = \sum a_i$ and $B= \sum b_j$, where the sums are with respect to the
  hyperaddition in $F$.  Since summation of more than 2 elements in $F$ is defined by the rule
  $a +_F b +_F c = \cup_{x\in a +_F b} x +_F c$, it follows immediately that $+_{\functor(F)}$
  restricts to a binary operation on $S(F)$.

Now suppose $F$ is doubly-distributive.  It follows from the doubly-distributive property that
\[
\left( \sum a_i \right) \times \left( \sum b_j \right ) = \sum a_i b_j
\]
and so $\times_{\functor}$ restricts to a binary operation on $S(F)$.  Clearly, the singletons
$\{0_F\}$ and $\{1_F\}$ become the additive identity and multiplicative identity.  It remains to
show multiplication distributes over addition, and this again follows directly from the doubly-distributive
property since if $A = \sum a_i$, $B = \sum b_j$, and $C= \sum c_k$ are subsets of $F$ then
\begin{align*}
  (A + B)\times C & = \left(\sum a_i + \sum b_j \right) \times \left(\sum c_k\right) \\
& = \sum a_i c_k +\sum b_j c_k\\
&  = ( A\times C) + (B\times C). 
\end{align*}

Given a strict homomorphism of doubly-distributive hyperfields, $f: F \to F'$, the induced mapping
$S(F) \to S(F')$ sends $A$ to $f(A)$.  A priori, $f(A)$ is an element of $\mathcal{P}(F')$, but
since $A =\sum a_i$ and $f$ is strict, $f(A) = \sum f(a_i)$, and so $f(A)$ is indeed an element of
$S(F')$.
\end{proof}

\begin{theorem}\label{reducedfunctortheorem}
  Let $F$ be a doubly-distributive hyperfield. Let $K=S(F)$,
  $K_0=\{A \subseteq S(F) \mid 0_F \in A\}$, $\varepsilon=-1_F$, $0=\{0_F\}$, and $1=\{1_F\}$. Then
  $\reducedfunctor(F) := (K;+,\times;\varepsilon,K_0)$ is a fuzzy ring.  Moreover, there is an
  inclusion $\reducedfunctor(F) \hookrightarrow \functor(F)$ that restricts to an isomorphism on
  units, and this defines a functor $\ddhyperfields \longrightarrow \fuzzystrong$.
\end{theorem}
\begin{proof}
The proof is essentially same as that of Theorem \ref{object} and  Proposition \ref{homomor}.
\end{proof}

As with $\functor$, the functor $\reducedfunctor$ is faithful since the hyperfield hypothesis
implies that all nonzero elements in $F$ are units and the groups of units in $F$,
$\reducedfunctor(F)$, and $\functor(F)$ are all canonically identified.

\begin{example}\label{dddd}
  It is clear that for $S(F)$ to be a semiring, the hyperfield $F$ must be doubly-distributive since
  any semiring satisfies the doubly-distributive property. For example, let
  $(F,\triangledown,\cdot)$ be the triangle hyperfield $\Delta$ (see, \cite[Theorem
  5.B]{viro}). Then one has $2\triangledown 3 = [1,5]$ and hence $(2\triangledown
  3)^2=[1,25]$. However, we have
  \begin{equation}\label{123}
    (2\cdot 2)\triangledown (2\cdot 3) \triangledown (3\cdot 2)\triangledown (3\cdot 3)=4\triangledown 6\triangledown 6 \triangledown 9 = [0,25]. 
  \end{equation}
  Now, for $A=\{2\}$, $B=\{3\}$, $C=\{2\}$, $D=\{3\}$, since $S(F)$ is a semiring, we should have
  \[
  (2+_S3)\times_S (2+_S 3) = (2\times_S 2)+_S (2\times_S 3) +_S (3\times_S 2)+_S (3\times_S 3)
  \]
  But from \eqref{123} we know this is not true.
\end{example}

\begin{remark}
\begin{enumerate}
\item Note that Example \ref{dddd} does not contradict the construction of Theorem \ref{object}
  since for fuzzy rings, one only requires that units distribute, rather than requiring all elements
  to distribute over sums.
\item When $F$ is not doubly-distributive, the set $S(F)$ need not be multiplicatively closed inside
  $\mathcal{P}(F)$.  In fact, it is multiplicatively closed whenever $F$ satisfies the following
  condition: for any $a_1, \ldots ,a_n, b_1, \ldots b_m \in R$ and $n \in \mathbb{N}$
  \begin{equation}\label{condicondi}
    \exists c_1 , \ldots, c_\ell \textrm{ such that } \left(\sum a_i\right) \times \left(\sum b_i\right) = \sum c_i.
  \end{equation}
\item One natural question is whether the condition \eqref{condicondi} is equivalent to the
  doubly-distributive property or not. The doubly-distributive property directly implies the
  condition \eqref{condicondi}, however the converse is not true in general. For example, the
  triangle hyperfield in Example \ref{dddd} satisfies the condition \eqref{condicondi}, but is not
  doubly-distributive.
\end{enumerate}
\end{remark}

We now prove Theorem D from the introduction.
\begin{theorem}
There is a natural isomorphism of fuzzy rings $\reducedfunctor(H_\Gamma) \cong K_\Gamma$.
\end{theorem}
\begin{proof}
  This is just a straightforward checking of the definitions.  The (hyperaddition) sum of two
  elements in $H_\Gamma$ is either a singleton or an interval of the form $[0, x]$, which gives
  exactly the set $K_\Gamma$.  Moreover, the sum in $H_\Gamma$ of two subsets of these types is
  easily seen to agree with the addition operation in $K_\Gamma$.  Thus the underlying set of
  $\reducedfunctor(H_\Gamma)$ is precisely the same as the underlying set of $K_\Gamma$ and the
  addition operations agree.  The multiplication operations clearly also agree.  The null elements
  of $\reducedfunctor(H_\Gamma)$ are the those subsets of $\Gamma\cup\{0\}$ that contain $0$, which
  is precisely the set of null elements of $K_\Gamma$.  In $K_\Gamma$, the element $\epsilon$ is the
  singleton $\{1\}$, and in $\reducedfunctor(H_\Gamma)$, it is the singleton $\{-1\}$, where
  $-1\in H_\Gamma$ is the the additive inverse of the identity element $1 \in \Gamma$, which is
  simply equal to $1$.
\end{proof}

\subsection{Factoring \texorpdfstring{$\reducedfunctor$}{F-bar} through partial demifields}

Let us first recall the definition of partial demifields from \cite[\S 4]{baker2016matroids}.
\begin{definition}
  A \textbf{partial demifield} is a pair $(F,S)$ consisting of a semiring $(S,+_S, \times_S)$ and a
  hyperfield $(F,+_F, \times_F)$
  such that
\begin{enumerate}
 \item $F$ is a submonoid of $S$ with respect to multiplication and $F$ generates $S$, i.e., the smallest
  subsemiring of $S$ which contains $F$ is $S$ itself;
\item For $a,b \in F$, if $a +_S b \in F$ then  $a +_S b \in a +_F b$.
\end{enumerate}
A morphism of partial demifields $(F,S) \to (F',S')$ is a semiring homomorphism $f: S \to S'$ that
restricts to a hyperfield homomorphism $F \to F'$.
\end{definition}

\begin{lemma}\label{partialdemi}
  Let $F$ be a doubly-distributive hyperfield. Then $(F,S(F))$ is a partial demifield.
\end{lemma}
\begin{proof}
  By identifying singletons of $S(F)$ with elements of $F$, one can consider $F$ as a
  (multiplicative) submonoid of $S(F)$. Also, by definition, $F$ generates $S(F)$ since the addition
  of singletons of $S(F)$ agrees with the hyperaddition of $F$. For the compatibility condition, let
  $a,b \in S(F)$ such that $a,b, a+_Sb \in F$. But, this means that $a,b$ are singletons such that
  $a+b$ is single-valued and hence $a+_Sb = a+b$.
\end{proof}

\begin{example}
  Let $\mathbb{S}:=\{-1,0,1\}$ be the hyperfield of signs. Then
  $S(\mathbb{S})=\{-1,0,1,\mathbb{S}\}$. Note that we identify singletons with elements of
  $\mathbb{S}$. One can easily see that the partial demifield $(\mathbb{S},S(\mathbb{S}))$ is
  isomorphic to the partial demifield $\hat{\mathbb{S}}$ which is constructed by Baker in
  \cite[Example 4.5.]{baker2016matroids}.
\end{example}

\begin{proposition}\label{partialdemifielsubcategory}
Let $(F,S)$ be a partial demifield. Suppose that $(F,S)$ satisfies the following condition: 
\begin{equation}\label{addsame}
a_1 +_F a_2 +_F \cdots +_F a_n = a_1 +_S a_2 +_S \cdots +_S a_n, \quad \forall a_i \in F, \quad \forall n \in \mathbb{N}.
\end{equation}
Then $(F,S)$ is isomorphic to the partial demifield $(F,S(F))$ constructed in Lemma \ref{partialdemi}. 
\end{proposition}
\begin{proof}
  First of all, it follows from \eqref{addsame} that $F$ is doubly-distributive and hence $(F,S(F))$
  is indeed a partial demifield. Now, any element $A \in S(F)$ is of the form
  $A=f_1+_Ff_2+_F\cdots +_Ff_n$ for some $f_i \in F$. Consider the the following map:
\[
f: S(F) \longrightarrow S, \quad A \mapsto f_1+_Sf_2+_S\cdots +_Sf_n. 
\] 
It follows from the assumption \eqref{addsame} that $f$ is well defined since if 
\[
A=f_1+_Ff_2+_F\cdots +_Ff_n=g_1+_Fg_2+_F\cdots +_Fg_m \quad \textrm{for some }f_i,g_i \in F,\quad n,m \in \mathbb{N} 
\]
then from the assumption \eqref{addsame} we have 
\[
f_1+_Sf_2+_S\cdots +_Sf_n=g_1+_Sg_2+_S\cdots +_Sg_m.
\]
Furthermore, this also shows that the map $f$ is an injective homomorphism of partial demifields. It
follows that $f(S(F))$ is a sub-semiring of $S$ which contains $F$. However, $S$ is the smallest
semiring which contains $F$. This implies that $f$ should be surjective as well and hence $f$ is an
isomorphism.
\end{proof}


Let $\ddhyperfields$ be the category of doubly-distributive hyperfields with strict homomorphisms
and $\partialdemifields$ be the category of partial demifields. For an object $F$ of
$\ddhyperfields$, we let $\demifunctorA(F)=(F,S(F))$ as in Lemma \ref{partialdemi}. For a morphism
$f:F\longrightarrow F'$ of $\ddhyperfields$, we let
$\demifunctorA(f):=S(f):S(F)\longrightarrow S(F')$ as in Theorem \ref{reducedfunctortheorem}.
\begin{proposition}
  $\demifunctorA$ is a faithful functor from $\ddhyperfields$ to $\partialdemifields$.
\end{proposition}
\begin{proof}
$\demifunctorA$ is clearly a functor and hence we only have to show that for any doubly-distributive hyperfields $A$ and $B$, 
\[
\Hom_{\ddhyperfields}(A,B)\longrightarrow \Hom_{\partialdemifields}(\demifunctorA(A),\demifunctorA(B))
\]
is injective.  Suppose that we have two strict homomorphisms $f, g:A \longrightarrow B$ such that
$\demifunctorA(f)=\demifunctorA(g)$. This, in particular, implies that
$\demifunctorA(f)|_A=\demifunctorA(g)|_B$ and hence $f=g$. This proves that the functor
$\demifunctorA$ is faithful.
\end{proof}

However the functor $\demifunctorA$ is not full as the following example shows. 

\begin{example}\label{f_1notfull}
  Let $\mathbb{K}$ be the Krasner hyperfield and $\mathbb{T}$ be the tropical hyperfield. Note that
  $\mathbb{T}$ is doubly-distributive, in fact, any hyperfield obtained from a linearly ordered
  abelian group (as in \S \ref{linearorderedhypergroup})) is doubly-distributive (see \cite[\S
  5.2]{viro}). One can easily see that there exists no strict homomorphism from $\mathbb{K}$ to
  $\mathbb{T}$. On the other hand, we have many morphisms from the partial demifield
  $(\mathbb{K},\demifunctorA(\mathbb{K}))$ to the partial demifield
  $(\mathbb{T},\demifunctorA(\mathbb{T}))$. For instance, one can define
\[
f:\demifunctorA(\mathbb{K})=\{0,1,\mathbb{K}\} \longrightarrow \demifunctorA(\mathbb{T}), \quad
f(0)=-\infty,\quad f(1)=0, \quad f(\mathbb{K})=\left[-\infty, 0\right]
\]
This shows that the functor $\demifunctorA$ is not full.
\end{example}

In fact, our construction can be considered as a `reduced' version of our previous functor
$\functor$ in $\S \ref{functor}$ in the following sense. We note that for any doubly-distributive
hyperfield $F$, $S(F)$ is a subset of $\mathcal{P}^*(F)$ since by definition $S(F)$ is the set of
all possible sums.

The functor $\reducedfunctor$ is not full. Indeed, one can easily observe that the hyperfields we
presented in Example \ref{notfullexample} are doubly-distributive and hence the same example shows
that $\reducedfunctor$ is not full as well.  Also, the following composition
\[
\ddhyperfields \stackrel{\reducedfunctor}{\longrightarrow} \fuzzystrong  \to \fuzzyweak
\]
is faithful but not full, as the following example shows.

\begin{example}\label{compositionnotfull}
  Let $\mathbb{Q}$ be the field of rational numbers (considered as a doubly-distributive hyperfield)
  and $\mathbb{K}$ be the Krasner hyperfield (which is doubly-distributive as we mentioned
  above). One can easily see that there is no strict homomorphism from $\mathbb{Q}$ to
  $\mathbb{K}$. On the other hand, since $\reducedfunctor(\mathbb{K})$ is the final object in
  $\fuzzyweak$ (since in this case, $\reducedfunctor(\mathbb{K})=\functor(\mathbb{K})$), there
  exists a weak morphism from $\reducedfunctor(\mathbb{Q})$ to $\reducedfunctor(\mathbb{K})$. This
  shows that the above composition is not full.
\end{example}

Let $\mathfrak{C}$ be the full subcategory of the category $\partialdemifields$ of partial
demifields whose objects satisfy the condition given in Proposition
\ref{partialdemifielsubcategory}; then it follows from Proposition \ref{partialdemifielsubcategory}
that $\mathfrak{C}$ is the essential image of the functor $\demifunctorA$.



Now, for each object $P=(F,S)$ in $\mathfrak{C}$, we let
$\demifunctorB(P)=\reducedfunctor(F)$. Also, for each homomorphism $f:P\longrightarrow P'$, we let
$\demifunctorB(f)=\reducedfunctor(f|_F)$.

\begin{proposition}
$\demifunctorB$ is a faithful functor from $\mathfrak{C}$ to $\fuzzystrong$. 
\end{proposition}
\begin{proof}
This is straightforward from the definition and above lemmas.  
\end{proof}

\begin{example}\label{funexample}
  Let $\Gamma$ be a linearly ordered abelian group. As in $\S \ref{review}$, we may assume that
  $\Gamma$ is equipped with the smallest element $-\infty$. Recall that $\Gamma$ can be enriched to
  a hyperfield $H_\Gamma$. As it was pointed out by Baker (see, \cite[Example
  4.4]{baker2016matroids}), $H_\Gamma$ is a demifield and hence a partial demifield. In particular,
  $H_\Gamma$ is an object of the category $\mathfrak{C}$. Now, one can easily see that
  $\demifunctorB(H_\Gamma)=K_{\Gamma}$, where $K_{\Gamma}$ is the fuzzy ring associated to $\Gamma$
  constructed by Dress and Wenzel (cf. \cite{dress1992valuated}, \cite{dress2011algebraic}).
\end{example}

We finish this section by illustrating these functors applied to the hyperfield of signs $\mathbb{S}$. 

\begin{example}
Let $\mathbb{S}=\{-1,0,1\}$ be the hyperfield of signs. We have the following:
\begin{enumerate}
\item
\[
\functor(\mathbb{S})=(K;+,\times;-1,K_0),
\]
where
\[ K=\{0,1,-1,\{0,1\},\{1,-1\},\{0,-1\},\mathbb{S}\}, \quad  K_0=\{0,\{0,1\},\{0,-1\},\mathbb{S}\}. 
\]
\item
\[
\reducedfunctor(\mathbb{S})=(K_{red};+,\times;-1,(K_0)_{red}),\textrm{ where }K_{red}=\{0,1,-1,\mathbb{S}\} \textrm{ and } (K_0)_{red}=\{0,\mathbb{S}\}. 
\]
\item
\[
\demifunctorA(\mathbb{S})=(\mathbb{S},S(\mathbb{S})),\quad S(\mathbb{S})=\{0,1,-1,\mathbb{S}\}. 
\]
\end{enumerate}
\end{example}

\section{From matroids over hyperfields to matroids over fuzzy rings, and back, via the functors \texorpdfstring{$\functor$}{F} and \texorpdfstring{$\invfunctor$}{G}}\label{sec:matroid-axioms}

In this section we investigate how the functors $\functor$ and $\invfunctor$ relates matroids over hyperfields, as
introduced by Baker \cite{baker2016matroids}, and matroids over fuzzy rings, as introduced by Dress
\cite{dress1986duality}.  There are numerous cryptomorphic axiom sets for ordinary matroids, and
many of these generalize to matroids with either hyperring or fuzzy ring coefficients.  We will only focus on the \emph{Grassmann-Pl\"ucker axioms} and show that both approaches are in fact equivalent via the functors $\functor$ and $\invfunctor$. 

\subsection{Grassmann-Pl\"ucker function axioms for matroids with coefficients}

We first recall the definition of a Grassmann-Pl\"{u}cker function of rank $r$ on a finite set
$E$ with coefficients in a hyperfield or a fuzzy ring.  

\begin{definition}[{Hyperfields Grassmann-Pl\"ucker functions \cite[Definition 3.9]{baker2016matroids}}]
  \label{hypgrass}
  Let $(F, \oplus, \odot)$ be a hyperfield.  A Grassmann-Pl\"{u}cker function of rank $r$ on a
  finite set $E$ with coefficients in $F$ is a function $\varphi:E^r \longrightarrow F$ such that:
\begin{enumerate}
\item[(GPH1)] $\varphi$ is not identically zero.
\item[(GPH2)] $\varphi$ is alternating.
\item[(GPH3)] (Grassmann-Pl\"{u}cker relations) For any two subsets $\{x_1,...x_{r+1}\}$ and
  $\{y_1,...,y_{r-1}\}$ of E, we have
\begin{equation}\label{hyppluc}
0_F \in \bigoplus_{k=1}^{r+1}(-1)^k\varphi(x_1,x_2,...,\hat{x_k},...,x_{r+1})\odot \varphi(x_k,y_1,...,y_{r-1}).
\end{equation}
\end{enumerate}
\end{definition}

\begin{definition}[{Fuzzy rings Grassmann-Pl\"ucker functions \cite[Definition
  4.1]{dress1991grassmann}}] \label{fuzzygrass} Let $K$ be a fuzzy ring with the group of multiplicative
  units $K^\times$.  A Grassmann-Pl\"{u}cker function of rank $r$ on a finite set $E$ with coefficients
  in $K$ is a function $\varphi:E^r \longrightarrow K^\times\cup \{0\}$ such that:
\begin{enumerate}
\item[(GPF1)]
$\varphi$ is not identically zero.
\item[(GPF2)]
$\varphi$ is $\varepsilon$-alternating, i.e, for any $x_1,...,x_r \in E$ and an odd permutation $\sigma \in S_r$, we have
\begin{equation}
\varphi(x_{\sigma(1)},...,x_{\sigma(r)})=\varepsilon\varphi(x_1,...,x_r),
\end{equation} 
and if the number of distinct elements in $\{x_1,...,x_r\}$ is smaller than $r$ then $\varphi(x_1,...,x_r)=0$.
\item[(GPF3)] (Grassmann-Pl\"{u}cker relations) For any two subsets $\{x_1,...x_{r+1}\}$ and $\{y_1,...,y_{r-1}\}$ of E, we have
\begin{equation}\label{fuzzypluc}
\sum_{k=1}^{r+1}\varepsilon^k\varphi(x_1,x_2,...,\hat{x_k},...,x_{r+1})\times \varphi(x_k,y_1,...,y_{r-1}) \in K_0.
\end{equation}
\end{enumerate}
\end{definition}

In both the hyperring case and the fuzzy ring case, we will be interested in equivalence classes of
Grassmann-Pl\"ucker functions, where two functions are equivalent if one is obtained by the other by
multiplication by a unit.

\subsection{Equivalence of the matroid theories}

Let $E$ be a non-empty finite set and $r$ be a positive integer. Let $F$ be a hyperfield and
$\functor(F)$ the associated fuzzy ring.  Recall that there is an identification of groups of
multiplicative units $F^\times \cong \functor(F)^\times$ given by sending $x$ to the singleton $\{x\}$.
Recall that there is also an identification $K^\times\simeq\mathcal G(K)^\times$ for every fuzzy ring $K$.

\begin{proposition}\label{onetoone}
Let $E$ be a finite set, $F$ be a hyperfield, $K$ be a field-like fuzzy ring, and $r$ a positive integer. 
\begin{enumerate}
\item
A function $\varphi: E^r \to F^\times \cong \functor(F)^\times$ is a Grassmann-Pl\"ucker function in the fuzzy ring sense (Definition \ref{fuzzygrass}) if and
  only if it is a Grassmann-Pl\"ucker function in the hyperring sense (Definition \ref{hypgrass}).
  If the hyperfield $F$ is  doubly-distributive then the same is true with $\reducedfunctor$ in
  place of $\functor$.
\item
A function $\varphi: E^r \to K^\times \cong \invfunctor(K)^\times$ is a Grassmann-Pl\"ucker function in the fuzzy ring sense (Definition \ref{fuzzygrass}) if and
  only if it is a Grassmann-Pl\"ucker function in the hyperring sense (Definition \ref{hypgrass}).
\end{enumerate}
\end{proposition}
\begin{proof}
  One sees that conditions (GPH1) and (GPH2) of Definition \ref{hypgrass} are equivalent (GPF1) and
  (GPF2) in Definition \ref{fuzzygrass} since $\varepsilon =-1_F$.  All that remains is to verify
  that (GPF3) and (GPH3) are equivalent, but this is immediately true since, by construction, the set of null elements in $\functor(F)$
  is
  \[
  \functor(F)_0=\{A \subseteq F \mid 0_F \in A\}.
  \]
  The same argument works verbatim for $\reducedfunctor$ and $\invfunctor$. 
\end{proof}

%
%

\begin{example}
  Let $\mathbb{K}$ be the Krasner hyperfield.  A matroid over $\mathbb{K}$ is the same thing as
  a matroid by \cite[\S 3]{baker2016matroids}. A matroid with coefficients in the fuzzy ring
  $\functor(\mathbb{K})$ is also the same thing as a matroid by \cite[\S 1.3]{dress1986duality}.
\end{example}

\begin{example}\label{sign}
  Let $\mathbb{S}$ be the hyperfield of signs. A matroid over $\mathbb{S}$ is the same thing as
  an oriented matroid (see, \cite[\S 3]{baker2016matroids}). Let $\functor(\mathbb{S})$ be the
  fuzzy ring associated to $\mathbb{S}$. Then $\functor(\mathbb{S})$ contains a sub-fuzzy
  ring $K:=\mathbb{R}\sslash \mathbb{R}^+$ (see \cite[\S 6]{dress1991grassmann}). In
  \cite{dress1991grassmann}, the authors showed that a matroid with coefficients in $K$ is the same
  thing as an oriented matroid. But, the proof only depends on Grassmann-Pl\"{u}cker functions on $E$
  with coefficients in $K$. However, since $K^\times=\functor(\mathbb{S})^\times$, the same proof shows
  that a matroid with coefficients in $\functor(\mathbb{S})$ is the same thing as an oriented matroid.
\end{example}

For a totally ordered abelian group $\Gamma$, Dress and Wenzel associated the fuzzy ring $K_{\Gamma}$
and proved that a valuated matroid (with a value group $\Gamma$) is the same thing as a matroid with
coefficients in $K_\Gamma$. For more details, see \cite{dress1992valuated} or $\S \ref{review}$.
\begin{example}
  Let $\mathbb{T}$ be the tropical hyperfield. A matroid over $\mathbb{T}$ is the same thing as a
  valuated matroid (see, \cite[\S 3]{baker2016matroids}). Let $\functor(\mathbb{T})$ be the
  fuzzy hyperring associated to $\mathbb{T}$. It follows from Proposition \ref{inj} that
  $K_\mathbb{R}$ is a sub-fuzzy ring of $\functor(\mathbb{T})$. Furthermore, we have
  $K_\mathbb{R}^\times=\functor(\mathbb{T})^\times$. Similar to Example \ref{sign}, one can see that a
  matroid with coefficients in $\functor(\mathbb{T})$ is the same thing as a matroid with
  coefficients in $K_\mathbb{R}$ and hence a valuated matroid.
\end{example}

\begin{example}
  Let $\mathbb{P}$ be the phase hyperfield. One can now similarly confirm that a matroid with
  coefficient in $\functor(\mathbb{P})$ is a complex matroid. This follows from the fact that a
  matroid with coefficients in the fuzzy ring $\mathbb{C}\sslash \mathbb{R}^+$ is a complex matroid and $\mathbb{C}\sslash \mathbb{R}^+$ is
  the sub-fuzzy ring of $\functor(\mathbb{P})$ such that
  $(\mathbb{C}\sslash \mathbb{R}^+)^\times=\functor(\mathbb{P})^\times$.
\end{example}

\bibliography{fuzzy-hyperring}
\bibliographystyle{alpha}

\end{document}